\documentclass[12pt]{amsart}
\usepackage{amsfonts}
\usepackage{amsfonts,latexsym,rawfonts,amsmath,amssymb,amsthm}
\usepackage[backend=biber]{biblatex}
\usepackage[plainpages=false]{hyperref}
\usepackage{graphicx}
\usepackage[title]{appendix}%

\numberwithin{equation}{section}

\RequirePackage{color}
\theoremstyle{plain}

\newtheorem{theorem}{Theorem}[section]

\newtheorem{corollary}{Corollary}[section]

\newtheorem{lemma}[theorem]{Lemma}
\newtheorem{proposition}[theorem]{Proposition}
\newtheorem{remark}[theorem]{Remark}

\newcommand{\beq}{\begin{equation}}
\newcommand{\eeq}{\end{equation}}
\newcommand{\beqs}{\begin{eqnarray*}}
\newcommand{\eeqs}{\end{eqnarray*}}
\newcommand{\beqn}{\begin{eqnarray}}
\newcommand{\eeqn}{\end{eqnarray}}
\newcommand{\beqa}{\begin{array}}
\newcommand{\eeqa}{\end{array}}

\def\p{\partial }

\def\ve{\varepsilon}
\def\n{\nabla}
\def\a{\alpha}
\def\b{\beta}
\def\g{\gamma}

\def\<{\langle}
\def\>{\rangle}
\def\Ric{\operatorname{Ric}}

\begin{document}
\title[Geometric inequalities for $p$-capacitary functions]{Mass--\(p\)-Capacity Inequalities in Asymptotically Flat Half-Spaces}

	\author{Chao Xia}
	\address{School of Mathematical Sciences\\
		Xiamen University\\
		361005, Xiamen, P.R. China}
	\email{chaoxia@xmu.edu.cn}
	\author{Jiabin Yin}
	\address{School of Mathematics and Statistics\\ Xingyang Normal University\\  541004, Xingyang, P.R. China}
	\email{jiabinyin@126.com}
	\author{Xingjian Zhou}
	\address{School of Mathematical Sciences\\
		Xiamen University\\
		361005, Xiamen, P.R. China}
	\email{zhouxingjian@stu.xmu.edu.cn}



%\thanks{This research was supported by funds from Hubei Provincial Department of Education
%Key Projects D20181003.}
%\thanks{$\ast$ Corresponding author}

\begin{abstract}
In this paper, we %establish the existence, regularity, and asymptotic behavior of a $p$-capacitary potential in an $3$-dimensional asymptotically flat half space.  Then, based on these properties of the potential, we 
establish general monotone quantities and sharp mass-capacity inequalities related to $p$-capacitary functions in $3$-dimensional asymptotically flat half-spaces of simple topology with nonnegative scalar curvature and nonnegative boundary mean curvature. These inequalities attain equality on  a Schwarzschild half-space outside a rotationally symmetric half sphere.
\end{abstract}

\maketitle

\baselineskip18pt

\parskip3pt

\section{Introduction}

{

%General relativity provides a fundamental link between gravitation and geometry. In the time-symmetric setting, an isolated gravitating system is modeled by a complete Riemannian manifold whose geometry at infinity reflects the absence of external matter. One models the exterior region of a single isolated body by a complete Riemannian manifold $(M,g)$ which is asymptotically flat. %: there exists a compact set $K\subset M$ such that $M\setminus K$ is diffeomorphic to $\mathbb{R}^3\setminus\Omega$ for some compact $\Omega\subset\mathbb{R}^3$, and in the associated chart at infinity the metric has the form
%$g_{ij}=\delta_{ij}+\sigma_{ij},$
%where $\sigma_{ij}$ and its derivatives up to second order satisfy suitable decay estimates. An appropriate notion of total mass for asymptotically flat manifolds was introduced by Arnowitt--Deser--Misner \cite{Arnowitt--Deser--Misner--1959}. 

%A fundamental question is whether nonnegative scalar curvature implies nonnegative total mass, and whether the zero-mass case forces Euclidean geometry. In dimension three this was proved by Schoen--Yau \cite{Schoen--Yau--1979} via minimal surface methods, yielding both positivity and rigidity.

%When the manifold has compact boundary, modeling an isolated body with horizon, one seeks sharper lower bounds for the mass. Motivated by cosmic censorship, Penrose \cite{Penrose--1973} proposed what is now called 
The Riemannian Penrose inequality says that,  for an asymptotically flat $3$-manifold $(M,g)$ with outermost minimal boundary and nonnegative scalar curvature,
\begin{equation}\label{RPI}
\mathfrak{m}_{ADM}\ge \sqrt{\frac{|\partial M|}{16\pi}},
\end{equation}
with equality if and only if $(M,g)$ is a spatial Schwarzschild manifold. The connected-boundary case was first proved by Huisken--Ilmanen \cite{Huisken--Ilmanen--2001} using weak level-set formulation of inverse mean curvature flow and Geroch-type monotonicity. The general case and the higher dimensional inequality in dimensions less than eight were subsequently proven by Bray \cite{Bray--2001} and  Bray-Lee \cite{Bray--Lee--2009} via conformal flow method. 
On the other hand, Bray \cite{Bray--2001} and Bray-Miao \cite{Bray--Miao--2008} established the following mass-capacity inequality  for an asymptotically flat manifold with  nonnegative scalar curvature and connected boundary, \begin{eqnarray}\label{mass-to-capacity}
		\frac{\mathfrak{m}_{ADM}}{{\rm Cap}(\p M)}\ge 1-\left(\frac{1}{16\pi}\int_{\p M}H^2\right)^{\frac12}.
	\end{eqnarray}

Recently, linear and nonlinear potential theory has been sucessfully applied to geometrical problem on $3$-manifolds of nonnegative scalar curvature, see for example \cite{Stern--2022, Bray--Kazaras--Khuri--Stern--2022, Agostiniani--Mazzieri--Oronzio--2022, Munteanu--Wang--2023, Munteanu--Wang--2026, Chan--Chu--Lee--Tsang--2024}, and to a resolution of $3$-dimensional stable Berstein problem \cite{Chodosh--Li--2024}. Regarding Riemannian Penrose inequality, Agostiniani-Mantegazza-Mazzieri-Oronzio \cite{Agostiniani--Mantegazza--Mazzieri--Oronzio--2023} provided a new proof for connected boundary by establishing monotonicity for $p$-capacitary functions in nonlinear potential theory. Also, Miao \cite{Miao--2025} found several monotone quantities associated to capacitary functions, so as to establish a family of sharp geometric inequalities, which lead to new proof of \eqref{mass-to-capacity} on $3$-manifolds with simple topology, see also \cite{Hirsch--Miao--Tam--2024}. Subsequently, we have established in \cite{ Xia--Yin--Zhou--2024} (See also \cite{Mazurowski--Yao--2025,Mazurowski--Yao--2026}) very general monotone quantities and sharp geometric inequalities related to $p$-capacitary functions, for which rigidity is modelled on Schwarzschild manifolds outside rotationally symmetric spheres. %Related monotonicity ideas also lead to mass--capacity inequalities and further geometric estimates; see, for example, \cite{ Agostiniani--Mazzieri--Oronzio--2024,Chan--Chu--Lee--Tsang--2024,Hirsch--Miao--Tam--2024,miao2023implications,  mazurowski2024mass, Munteanu--Wang--2023,Oronzio--2025}.

%More broadly, level-set monotonicity originates in work of Colding \cite{Colding--2012} and Colding--Minicozzi \cite{Colding--Minicozzi--2013, Colding--Minicozzi--2014} on Green's functions and harmonic functions under nonnegative Ricci curvature, and was subsequently used to derive rigidity and sharp inequalities under Ricci assumptions, for instance in \cite{Agostiniani--Fogagnolo--Mazzieri--2020,Agostiniani--Mazzieri--2020, Agostiniani--Mazzieri--2017,Benatti--Fogagnolo--Mazzieri--2024--AnalPDE,  Benatti--Fogagnolo--Mazzieri--2024--MathAnn}. 

%On the other hand, by combining level-set ideas with the Bochner formula, the Gauss equation, and the Gauss--Bonnet theorem, Stern \cite{Stern--2022} and Bray--Kazaras--Khuri--Stern \cite{Bray--Kazaras--Khuri--Stern--2022} showed in dimension three that one can obtain rigidity and mass-type information under the weaker assumption of nonnegative scalar curvature. These scalar-curvature based approaches have been further developed using harmonic and nonlinear potential theory; see, for instance, \cite{Agostiniani--Mazzieri--Oronzio--2024, Benatti--Fogagnolo--Mazzieri--2023, Benatti--Fogagnolo--Mazzieri--2025}. These developments provide a flexible framework for geometric inequalities in asymptotically flat manifolds.

{In this paper, we aim to develop a similar nonlinear potential-theoretic
approach to geometric problems on asymptotically flat half-spaces with
nonnegative scalar curvature and nonnegative boundary mean curvature, and to
extend our previous result \cite{Xia--Yin--Zhou--2024} to this setting.
% This extension is not merely formal, as the presence of a noncompact boundary introduces several new analytic difficulties. In particular, one has to establish the precise asymptotic behavior of the function $u$ appearing in Theorem~\ref{lem:2.2}, and to  establish suitable Sobolev inequalities adapted to manifolds with noncompact boundary.
}
% In this paper, we aim at developing a similar nonlinear potential theoretic approach for geometric problems in asymptotically flat {half-spaces} with nonnegative scalar curvature and nonnegative boundary mean curvature, and extending our previous result \cite{Xia--Yin--Zhou--2024} to this setting. {\color{red}This extension is not trivial and requires a new ingredient.  We need to resolve the analytic difficulties arising from non-compact boundaries, especially the asymptotic behavior of $u$ in Theorem \ref{lem:2.2} and Sobolev inequalities on manifolds with non-compact boundaries, among others.}
Let $(M,g)$ be a complete $3$-dimensional Riemannian manifold with non-compact boundary. We say that $(M,g)$ is a one-end asymptotically flat half-space if there exists a compact subset $\Omega\subset M$ such that $M\setminus\Omega$ is diffeomorphic to $\{x\in\mathbb{R}^3_+: |x|\ge 1\}$, where $\mathbb{R}^3_+=\{x\in\mathbb{R}^3:\ x_3\ge0\}$, and in the corresponding chart at infinity,
\[
|g_{ij}-\bar g_{ij}|+\sum_{k=1}^3 |x|\,|\partial_k g_{ij}|+\sum_{k,l=1}^3 |x|^2\,|\partial_k\partial_l g_{ij}|
=O(|x|^{-\tau}),\qquad \tau>\frac12,
\]
together with
\[
\Ric_{ij}\ge -O(|x|^{-2})
\qquad\text{as }|x|\to\infty,
\]
where $\bar g$ is the Euclidean metric on $\mathbb{R}^3_+$. 

 Almaraz--Barbosa--de Lima \cite{Almaraz--Barbosa--deLima--2016} has introduced the following mass-type quantity for an asymptotically flat half-space $M$
\begin{equation*}
\mathfrak{m}_{ABL}
:= \frac{1}{16\pi}\lim_{r\to\infty}
\left\{
\int_{S_r(0)\cap\mathbb{R}^3_+}
\bigl(\partial_j g_{ij}-\partial_i g_{jj}\bigr)\frac{x^i}{|x|}\,dA_{\bar g}
+\int_{S_r(0)\cap(\mathbb{R}^2\times\{0\})}\frac{x^i}{|x|}g_{i3}\,ds_{\bar g}
\right\}.
\end{equation*}
$\mathfrak{m}_{ABL}$ is well-defined, provided that the scalar curvature $R$ and the boundary mean curvature $H_{\p M}$ are both integrable.

A special asymptotically flat half-space is the Schwarzschild half-space of $\mathfrak{m}_{ABL}=m$, given by
\[
(M, g_m)=\left(\mathbb{R}^3_+\setminus B^3_{+, m}(0),(1+{\frac{m}{|x|}})^4\bar g\right),
\]
where $B^3_{+,r}(0)=\{x\in\mathbb{R}^3_+:\ |x|\le r\}$.

Almaraz--Barbosa--de Lima \cite{Almaraz--Barbosa--deLima--2016} proved the positive-mass theorem, saying that $\mathfrak{m}_{ABL}\ge0$ provided $R\ge0$ and $H_{\p M}\ge0$, with equality holding if and only if $(M,g)$ is isometric to the Euclidean half-space. Further aspects of mass and asymptotic geometry in asymptotically flat manifolds were studied by Almaraz--de Lima \cite{Almaraz--deLima--2023}.

Given a set $U\subset M$, we use
$\tilde{\p}U=\overline{\p U\setminus\p M} \hbox{ and }\hat{\p}U=\p U\cap\p M$
to denote the interior boundary and the exterior boundary of $U$ respectively. A compact connected hypersurface $\Sigma$ with non-empty boundary is called a {free boundary surface} if $\Sigma\cap\partial M=\partial\Sigma$ and $\Sigma$ meets $\partial M$ orthogonally. An unbounded,
connected subset $M' \subset M$ is called an exterior region if $\tilde{\p}M'$ consists of
closed and free boundary minimal surfaces and if $M'$ contains no other
closed or free boundary minimal surfaces.
Riemannian Penrose inequality for $\mathfrak{m}_{ABL}$ has been established by Koerber \cite{Koerber--2023} for $n=3$ and connected horizon boundary via free boundary weak inverse mean curvature flow, and by Eichmair--Koerber \cite{Eichmair--Koerber--2023} using a doubling procedure in dimensions $3\le n\le7$. It says that for an asymptotically flat half-space $M$ such that $R\ge 0$ and $H_{\p M}\ge 0$, and $M'\subset M$ an exterior region such that $\Sigma$ is a free boundary
component of $\tilde{\p }M'$, it holds that
\[
\mathfrak{m}_{ABL} \ge \left(\frac12\right)^{\frac n{n-1}} \left( \frac{|\Sigma|}{\omega_{n-1}} \right)^{\frac{n-2}{n-1}},
\]
with equality if and only if $(M',g)$ is isometric to the Schwarzschild half-space. 
Following Bray-Miao \cite{Bray--Miao--2008}, mass-capacity inequalities for $\mathfrak{m}_{ABL}$ have been recently investigated by Silva \cite{Silva--2025} and the second and third authors \cite{yin2026sharp}.

From now on, we do not assume $M'$ is an exterior region, and just let $M'$ be the interior of the unbounded component of $M\setminus\Sigma$, where $\Sigma$ is a free boundary surface. Then $
\partial M'=\hat\partial M'\cup\Sigma.$
For $1<p<3$, we consider the $p$-capacity of $\Sigma$ in $M'$,
\[
\mathrm{Cap}_{p}(\Sigma, M')
:=\inf\left\{
\int_{M'}|\nabla \varphi|^p
\;\middle|\;
\varphi\in\mathcal A
\right\},
\qquad
\mathcal A:=\{\varphi=f+v:\ v\in W^{1,p}_{0,\Sigma}(M')\},
\]
where $f-1\in C^\infty_c(\overline{M'})$ and $f=0$ near $\Sigma$ (see \cite[p.~480]{Salsa--2016} for $W^{1,p}_{0,\Sigma}(M')$). The minimizer is characterized by the $p$-capacitary potential $u$ that solves: (see section 2 for details)
\begin{equation}\label{p-H}
\left\{
\begin{array}{ll}
\Delta_p u=0 & \text{in } M',\\
u=0 & \text{on } \Sigma,\\
\langle\nabla u,\mu\rangle=0 & \text{on } \hat\partial M',\\
u(x)\to1 & \text{as } |x|\to\infty,
\end{array}
\right.
\end{equation}
where $\Delta_pu=div(|\nabla u|^{p-2}\nabla u)$. In particular,
\[
\operatorname{Cap}_p(\Sigma,M')
=\int_{M'}|\nabla u|^p dV_{M'}
=\int_{\{u=t\}}|\nabla u|^{p-1}d\sigma
\]
for regular level sets $\{u=t\}$.

We set
\[
a:=\frac{3-p}{p-1}>0,
\qquad
\mathfrak{c}_{p}:=\left(\frac{\operatorname{Cap}_{p}}{2\pi}\right)^{\frac{1}{p-1}}.
\]
In Section $2$, we get the following  asymptotic expansion of $u$:
\begin{equation}\label{equ:asymptotic u}
u=1-\frac{\mathfrak{c}_{p}}{a}r^{-a}+O_2(r^{-a-\tilde\tau})
\qquad\text{as }r=|x|\to\infty,
\end{equation}
for any $0<\tilde\tau<\min\{\tau,1\}$.

Given $m\in\mathbb{R}$ and $r_0\ge |m|>0$, we denote by $(\mathcal{M}^{+,3}_{m,r_0},g_m)$ the spatial Schwarzschild half-space of mass $m$ outside a half-ball,
\begin{equation}\label{schwarzschild}
\mathcal{M}^{+,3}_{m,r_0}:=\mathbb{R}^3_+\setminus B^3_{+,r_0}(0),
\qquad
g_m=\left(1+\frac{m}{|x|}\right)^4\bar g.
\end{equation}
 Note that $\partial B^3_{+,r_0}=S^2_{+,r_0}\cup S_{r_0}$, where $S^2_{+,r_0}=\{x\in\mathbb{R}^3_+:\ |x|=r_0\}$ and $S_{r_0}=\{x\in\partial\mathbb{R}^3_+:\ |x|=r_0\}$. If $m>0$ and $r_0=m$, then $S_{r_0}$ is totally geodesic and $S^2_{+,r_0}$ is minimal.

The main result of this paper is the following sharp geometric inequalities between $\mathfrak{m}_{ABL}$ and several integral quantities involving the $p$-capicitary potential.

\begin{theorem}\label{Thm:1.02}
Let $(M',g)$ be a $3$-dimensional, complete, one-end asymptotically flat half space with boundary $\partial M'=\hat\partial M'\cup\Sigma$ that has nonnegative scalar curvature and $H_{\hat\partial M'}\geq0$, where $\Sigma$ is a free boundary. Assume that $M'$ is simply connected and $\Sigma$ is connected. Let $p\in(1,3)$ and $u$ be the weak solution to \eqref{p-H}.
{For any $k\in(-1,1)$, denote
\[
m:= \operatorname{sgn}(k)\bigl(I_a(k)\mathfrak{c}_{p}\bigr)^{\frac{1}{a}},
\qquad
r_{0}:=\frac{m}{k}
=|k|^{-1}\bigl(I_a(k)\mathfrak{c}_{p}\bigr)^{\frac{1}{a}},
\]
where $a:=\frac{3-p}{p-1}$ and
\[
I_{a}(k):= \int_{0}^{|k|} s^{a-1}(1+\operatorname{sgn}(k)s)^{-2a}\,ds.
\]}
Then the following inequalities hold:
\begin{equation}\label{geom-ineq-10}
\begin{aligned}
& 2 \pi - \frac{(1+k)^{2}}{(1-k)^{2}} (\eta(r_{0}))^{2} \int_{\Sigma} |\nabla u|^2
\\ \ge\;& \frac{(1+k)^{2}r_{0}}{\eta(r_{0})} \frac{(1-a\eta(r_{0}))}{2m}
\left\{
2 \pi \frac{(1-k)^{2}}{(1+k)^{2}} - \int_{\Sigma} \left(\frac{H}{2}\right)^{2}
+ \int_{\Sigma} \left( \frac{H}{2}-\eta(r_{0})|\nabla u| \right)^{2}
\right\},
\end{aligned}
\end{equation}
\begin{equation}\label{geom-ineq-20}
\begin{aligned}
& 2 \pi -\frac{(1+k)^{2}}{(1-k)^{2}} (\eta(r_{0}))^{2}  \int_{\Sigma} |\nabla u|^2
\le 4\pi a\left( \mathfrak{m}_{ABL}-m \right) \frac{(1-a\eta(r_{0}))}{m} .
\end{aligned}
\end{equation}
Here $\eta(t)$ is defined by \eqref{equ:eta} below and
\[
\eta(r_{0})=\mathfrak{c}_{p}^{-1}r_{0}^{a}(1+k)^{2a-1}(1-k).
\]
Moreover, equality in each of the above inequalities holds for some $k$ if and only if $(M',g)$ is isometric to either the Schwarzschild half space of mass $m$ outside a rotationally symmetric half ball, $(\mathcal{M}^{+,3}_{m,r_0},g_m)$ with $m=r_0k\ (k\neq0)$, or the Euclidean half space outside a rotationally symmetric half ball $(\mathbb{R}^3_+\setminus B^3_{+,r_0}(0),\delta)$.
\end{theorem}

\begin{remark}
\item[(i)] When $k=0$, the quantities are understood in the limit sense, that is,
\[
I_{a}(0)=0,\quad m=0,\quad r_{0}=\left(\frac{\mathfrak{c}_{p}}{a}\right)^{\frac{1}{a}},
\quad \eta(r_{0})=\frac{1}{a},\quad
\lim_{k\to 0}\frac{1-a\eta(r_{0})}{m}=\frac{2}{a+1}\left(\frac{\mathfrak{c}_{p}}{a}\right)^{-\frac{1}{a}}.
\]
\item[(ii)] When $-1<k<0$, equality holds for $(\mathcal{M}^{+,3}_{m,r_0},g_m)$ with $m<0$.
\end{remark}

{When $k=1$, some quantities are also understood in the limit sense in Theorem~\ref{Thm:1.02}, like $\frac{\eta(r_0)}{1-k}=I_{a}(1)2^{2a-1}$ where $r_0=m$, then we have the following theorem.}

\begin{theorem}\label{Thm:1.01}
Under the assumptions of Theorem \ref{Thm:1.02}. Let $p\in (1, 3)$ and $u$ be the weak solution to \eqref{p-H}. Denote $m=\left(I_a(1)\mathfrak{c}_{p}\right)^{\frac{1}{a}}.$
Then we have
\begin{equation}\label{geom-ineq-1}
\begin{aligned}
&2 \pi
+ 4 \int_{\Sigma} H |\nabla u|
- 2^{4a} (I_{a}(1))^{2} \int_{\Sigma} |\nabla u|^2
\ge 0,
\end{aligned}
\end{equation}
\begin{equation}\label{geom-ineq-2}
\begin{aligned}
&2 \pi (1+2a)
- 2^{4a} (I_{a}(1))^{2} \int_{\Sigma} |\nabla u|^2
\le 4\pi a \frac{\mathfrak{m}_{ABL}}{m}.
\end{aligned}
\end{equation}
Moreover, equality in each of the above inequalities holds if and only if $(M',g)$ is isometric to the Schwarzschild half space of mass $m$, $(\mathcal{M}^{+,3}_{m},g_m)$.
\end{theorem}

From Theorem \ref{Thm:1.01} and Theorem \ref{Thm:1.02}, we get the following Bray--Miao-type mass-to-capacity inequality and capacity-to-area inequality.

\begin{theorem}\label{Thm:1.04}
Under the assumptions of Theorem \ref{Thm:1.02}. 
{Let $k\in (-1, 1]$ be such that
\[
1 - \frac{1}{8\pi} \int_{\Sigma} H^{2}  = \frac{4k}{(1+k)^{2}}.
\]
Then we have
\begin{eqnarray}
&&\mathfrak{m}_{ABL} \ge  \operatorname{sgn}(k) \left(I_a(k)\mathfrak{c}_{p}\right)^{\frac{1}{a}},\label{geom-ineq-4}\\
&&\sqrt{\frac{|\Sigma|}{32\pi}} \ge \frac{(1+k)^{2}}{4|k|} \left( I_{a}(k) \mathfrak{c}_{p} \right)^{\frac{1}{a}}.\label{geom-ineq-5}
\end{eqnarray}}
Moreover, equality in each of the above inequalities holds if and only if $(M',g)$ is isometric to the Schwarzschild half space of mass $m=\operatorname{sgn}(k) \left(I_a(k)\mathfrak{c}_{p}\right)^{\frac{1}{a}} (k\neq0)$ outside a rotationally symmetric half ball, or to $(\mathbb{R}^3_+\setminus B^3_{+,r_0}(0),\delta)$.
\end{theorem}

{
\begin{remark}
By removing the nonnegative Hawking mass assumption, our results improve both the mass--capacity inequality of Silva~\cite{Silva--2025} and the mass--$p$-capacity inequality of the second and third authors~\cite{yin2026sharp}.
\end{remark}}

When $k=1$, we have the following theorem.

\begin{theorem}\label{Thm:1.03}
Under the assumptions of Theorem \ref{Thm:1.02}. Assume that $\Sigma$ is minimal. Then we have
\begin{eqnarray}
&&\mathfrak{m}_{ABL} \ge  \left( I_{a}(1) \mathfrak{c}_{p} \right)^{\frac{1}{a}}, \label{geom-ineq-6}\\
&&\sqrt{\frac{|\Sigma|}{32\pi}} \ge \left( I_{a}(1) \mathfrak{c}_{p} \right)^{\frac{1}{a}}. \label{geom-ineq-7}
\end{eqnarray}
Moreover, equality in each of the above inequalities holds if and only if $(M',g)$ is isometric to the Schwarzschild half space of mass $m=\left( I_{a}(1) \mathfrak{c}_{p} \right)^{\frac{1}{a}}$, $(\mathcal{M}^{+,3}_{m},g_m)$.
\end{theorem}

As a consequence, let $p\to1$, we obtain the Riemannian Penrose-type inequality.
\begin{theorem}\label{Thm:1.6}
Under the assumptions of Theorem \ref{Thm:1.03}. Assume that $M'$
contains no other closed or free boundary minimal surfaces. Then we have
\begin{eqnarray}
&&\mathfrak{m}_{ABL} \ge  \sqrt{\frac{|\Sigma|}{32\pi}}. \label{geom-ineq-8}
\end{eqnarray}
Moreover, equality  can be obtained by the Schwarzschild half-space.
\end{theorem}

We prove Theorem \ref{Thm:1.01} via exhibiting monotonic quantities for $p$-harmonic functions.
To illustrate the monotonicity quantity, we shall use the following three one-variable functions. For any $k\in(-1,0)\cup(0,1]$,
let $m=\operatorname{sgn}(k)\left(I_a(k)\mathfrak{c}_{p}\right)^{\frac{1}{a}}$. We set
\begin{equation}\label{ode-solution1}
\begin{aligned}
\alpha(t)
=&\, t\left(1+\frac{m}{t}\right)^{2}
\left\{ \left( C_{2}\mathfrak{c}_{p}+C_{1}\frac{a}{2m} \frac{I_{a}(\frac{m}{t})}{I_{a}(k)}  \right)
\eta(t)-C_{1}\frac{1}{2m} \right\},\\
\beta(t)
=&\, - \eta(t)\alpha(t)
+ \left( C_{2}\mathfrak{c}_{p}\frac{2m}{a}+C_{1} \frac{I_{a}(\frac{m}{t})}{I_{a}(k)} \right) \mathfrak{c}_{p}^{-2} t^{2a}\left( 1+\frac{m}{t}\right)^{4a},\\
\gamma(t)
=&\, - \mathfrak{c}_{p}^{2} t^{-2a}\left( 1+\frac{m}{t}\right)^{-4a} \eta(t) \alpha(t)
-  \left( C_{2}\mathfrak{c}_{p}\frac{2m}{a}+C_{1} \frac{I_{a}(\frac{m}{t})}{I_{a}(k)} \right),
\end{aligned}
\end{equation}
where $C_1,C_2\in\mathbb{R}$ and $\eta$ is given by
\begin{equation}\label{equ:eta}
\eta(t):=\mathfrak{c}_{p}^{-1}t^{a}\left(1 + \frac{m}{ t}\right)^{2a-1}\left(1- \frac{m}{t}\right).
\end{equation}

We remark that $\alpha,\beta,\gamma$ satisfy the following ODE system (see Proposition~A.1 of \cite{Xia--Yin--Zhou--2024}):
\begin{equation}\label{equ: differential equations}
\begin{cases}
\alpha'(t) - (2 a + 1) \eta(t) f'(t) \alpha(t) - a f'(t) \beta(t) = 0, \\
\beta'(t) + (2 a + 1) (\eta(t))^2 f'(t) \alpha(t) = 0, \\
\gamma'(t) = - f'(t) \alpha(t),
\end{cases}
\end{equation}
where
\begin{equation}\label{equ:f}
f(t)
:= 1 - \int_{t}^{\infty} \mathfrak{c}_{p} s^{-a-1}\left(1+\frac{m}{s}\right)^{-2a}\,ds
=1-\frac{I_{a}(\frac{m}{t})}{I_{a}(k)}.
\end{equation}

Finally, we have the following monotone quantity along regular level sets of $p$-harmonic functions.

\begin{theorem}\label{thm:Monot}
Under the assumptions of Theorem \ref{Thm:1.02}.  Let $p\in(1,3)$ and $u$ be the weak solution to \eqref{p-H}.
{For any $k\in(-1,0)\cup(0,1]$}, let $\alpha,\beta,\gamma$ be given by \eqref{ode-solution1} with
\begin{equation}\label{alpha>0-assumpt}
C_{2} \ge 0,\qquad
\left( C_{2}\mathfrak{c}_{p}+C_{1}\frac{a}{2m} \right)\eta(r_{0}) \ge C_{1}\frac{1}{2m}.
\end{equation}
Let $F:[r_{0},\infty)\to\mathbb{R}$ be defined by
\begin{equation}\label{equ: F}
F(t):=2\pi\gamma(t)+\alpha(t)\int_{\Sigma_t}H|\nabla u|+\beta(t)\int_{\Sigma_t}|\nabla u|^2,
\end{equation}
where $\Sigma_t$ is a regular level set of $u$ given by
\[
\Sigma_t=\{x\in M:\ u(x)=f(t)\}.
\]
Then $F(t)$ is monotone non-increasing on
\[
\mathcal{T}:=\left\{ t\in[r_{0},\infty)\ \big|\ f(t)\ \text{is a regular value of }u \right\}.
\]
Moreover, $F$ is constant on $\mathcal{T}$ if and only if $(M',g)$ is isometric to the Schwarzschild half space of mass $m$ outside a rotationally symmetric half ball, $(\mathcal{M}^{+,3}_{m,r_0},g_m)$.
\end{theorem}

{
The proof of above theorems follows closely our previous work \cite{Xia--Yin--Zhou--2024} for asymptotically flat $3$-manifolds with compact boundary.  The presence of a noncompact boundary gives rise to a few new analytic difficulties. In particular, one needs to establish the precise asymptotic behavior of $p$-capacitary functions with mixed boundary value problem,  see Lemma~\ref{lem:2.2}.  %, as well as suitable Sobolev inequalities adapted to asymptotically flat half-spaces.

The paper is organized as follows. In Section~\ref{Sec:02}, we prove the existence and regularity of $p$-capacity potentials and derive the asymptotic expansion of the corresponding potential function. In Section~\ref{Sec:03}, we collect some facts about the Schwarzschild half-space, which will be used in the construction of monotone quantities. In Section~\ref{Sec:04}, we introduce these monotone quantities, and in Section~\ref{Sec:05}, we analyze their asymptotic behavior. Section~\ref{Sec:06} is devoted to the proof of the main results, including the mass--$p$-capacity inequalities. Finally, in Section~\ref{Sec:07}, we establish Sobolev inequalities on asymptotically flat manifolds with noncompact boundary and combine them with the mass--$p$-capacity inequality to prove the Penrose inequality.
}

} % end blue

\section{$p$-Capacitary functions with mixed boundary value problem}\label{Sec:02}

In this section, we will prove the existence, regularity, and asymptotic behavior of a $p$-capacitary potential of $3$-dimensional asymptotically flat half-space. Also we will prove the connectedness of the regular level sets of the solution.

\begin{lemma}
There exists a unique solution $u \in W^{1, p}_{loc}(M^{\prime})$ to the problem \eqref{p-H}
in the distribution sense. Moreover, $0<u\leq 1$ and
\begin{eqnarray*}
\mathrm{Cap}_{p}(\Sigma, M^{\prime})
:=\inf\bigg\{ \int_{M^{\prime}}|\nabla u|^p\bigg\}.
\end{eqnarray*}
\end{lemma}

\begin{proof}
\textbf{Step 1:} We obtain a local version of the result.

Assume that $\Phi: \{x \in \mathbb{R}^{3}_{+}: |x|\geq 1\}\rightarrow M\setminus \Omega$
is diffeomorphic.
Set $M^{\prime}_{R}:=\Phi(\{x \in \mathbb{R}^{3}_{+}: 1\leq|x|<R\})\cap M^{\prime}$
and $\Gamma_{R}:=\Phi(\{x \in \mathbb{R}^{3}_{+}: |x|=R\})$.
We consider
\begin{eqnarray*}
\mathrm{Cap}_{p}(\Sigma, M^{\prime}_{R})
:=\inf\bigg\{ \int_{M^{\prime}_{R}}|\nabla \varphi|^p: \varphi \in \mathcal{A}_R\bigg\},
\end{eqnarray*}
where
\begin{eqnarray*}
\mathcal{A}_R:=\Big\{\varphi: \varphi-f \in W^{1, p}_{0, \Gamma_{R}\cup \Sigma}(M^{\prime}_{R})\Big\}.
\end{eqnarray*}
Then, we choose a minimizing
sequence $\{\varphi_j\}\subset  \mathcal{A}_R$ such that
\begin{eqnarray*}
\lim_{j\rightarrow +\infty}\int_{M^{\prime}_{R}}|\nabla u_j|^p=\mathrm{Cap}_{p}(\Sigma, M^{\prime}_{R}).
\end{eqnarray*}
Applying the Poincar\'e inequality (see (7.69) in \cite{Salsa--2016}),
there exists a constant $C(R)$ such that
\begin{eqnarray*}
\int_{M^{\prime}_{R}}|u_j-f|^p\leq C(R)\int_{M^{\prime}_{R}}|\nabla (u_j-f)|^p
\leq C(R)(\int_{M^{\prime}_{R}}|\nabla u_j|^p+1)
\end{eqnarray*}
which implies that $\parallel u_j-f\parallel_{W^{1, p}_{0, \Gamma_{R}\cup \Sigma}}$ is bounded.
So, there exists a subsequence (which we still denote as $\{u_j\}$ for convenience)
such that
\begin{eqnarray*}
u_j-f\rightharpoonup u_R-f \quad \mbox{in} \quad W^{1, p}_{0, \Gamma_{R}\cup \Sigma}(M^{\prime}_{R})
\quad \mbox{as} \quad j\rightarrow +\infty.
\end{eqnarray*}
Then, the weak convergence implies that
\begin{eqnarray*}
\int_{M^{\prime}_{R}}|\nabla u_R|^p\leq
\lim_{j\rightarrow +\infty}\int_{M^{\prime}_{R}}|\nabla u_j|^p
=\mathrm{Cap}_{p}(\Sigma, M^{\prime}_{R}).
\end{eqnarray*}
Thus, $u_R \in \mathcal{A}_R$ and
\begin{eqnarray}\label{Yin11}
\mathrm{Cap}_{p}(\Sigma, M^{\prime}_{R})=\int_{M^{\prime}_{R}}|\nabla u_R|^p.
\end{eqnarray}
Moreover, by calculus of variation, it is easy to see that $u_R$ is a weak to
the local version of the problem \eqref{p-H}
\begin{equation*}
\left\{
\begin{array}{ll}
&\Delta_pu=0
\quad \mbox{in} \quad M^{\prime}_{R}, \\
&u=0 \quad \mbox{on} \quad \Sigma, \\
&\langle \nabla u, \mu\rangle=0 \quad \mbox{on} \quad \hat{\partial} M^{\prime}_{R},
\\
&u(x)=1 \quad \mbox{on} \quad \Gamma_R.
\end{array}
\right.
\end{equation*}
By applying \cite[Lemma 2.2]{Ma--Yang--Yin--2025}, we know that $u_r(x)\geq u_s(x)$ for $r>s$ and
$x \in M^{\prime}_{R}$ and $0\leq u_R\leq 1$.

\textbf{Step 2:} We find a weak solution to the problem \eqref{p-H}.

From the monotonicity and boundedness of solutions as mentioned above,
the function $u(x):=\lim_{R\rightarrow +\infty}u_R(x)$ is well defined for
$x \in M^{\prime}_{R}$. We will prove that $u$ is a unique weak solution
to the problem \eqref{p-H}.

We claim that
\begin{eqnarray}\label{Yin13}
\lim_{R\rightarrow +\infty}\mathrm{Cap}_{p}(\Sigma, M^{\prime}_{R})=\mathrm{Cap}_{p}(\Sigma, M^{\prime}).
\end{eqnarray}
On the one hand, since $\mathcal{A}_R\subset \mathcal{A}$, we have
\begin{eqnarray}\label{Yin14}
\mathrm{Cap}_{p}(\Sigma, M^{\prime}_{R})\geq\mathrm{Cap}_{p}(\Sigma, M^{\prime}).
\end{eqnarray}
On the other hand, for any function $v=u-f \in W^{1, p}_{0, \Sigma}(M^{\prime})$,
using Lemma 7.86 in \cite{Salsa--2016}, there exists a sequence of
$\{v_j\}\subset C_{c}^{\infty}(\overline{M^{\prime}})$ vanishing in a neighborhood of $\overline{\Sigma}$ which converges to $v$
in $W^{1, p}_{0, \Sigma}(M^{\prime})$. When $j$ is fixed, $u_j=f+v_j \in \mathcal{A}_R$
for $R$ large enough. Thus, we have
\begin{eqnarray*}
\int_{M^{\prime}_{R}}|u_j|^p\geq \mathrm{Cap}_{p}(\Sigma, M^{\prime}_{R})
\geq \lim_{R\rightarrow+\infty}\mathrm{Cap}_{p}(\Sigma, M^{\prime}_{R}).
\end{eqnarray*}
Since $u_j$ converges to $u$
in $W^{1, p}_{0, \Sigma}(M^{\prime})$, we obtain
\begin{eqnarray*}
\int_{M^{\prime}}|u|^p\geq \mathrm{Cap}_{p}(\Sigma, M^{\prime}_{R})
\geq \lim_{R\rightarrow+\infty}\mathrm{Cap}_{p}(\Sigma, M^{\prime}_{R}).
\end{eqnarray*}
Thus, the claim \eqref{Yin13} follows.

From \eqref{Yin11} and \eqref{Yin13}, and the interior regularity theory in \cite[Theorems 1 and 2]{DiBenedetto--1983--NATMA},
we deduce that $u_R$ is uniformly
bounded for large $R$ in $C^{1, \alpha}_{loc}(M^{\prime}_{R})\cap W^{1, p}(M^{\prime}_{R})$.
By Arzel\`a-Ascoli
theorem and a diagonal process one can find a sequence $R_j\rightarrow +\infty$ such that
$u_{R_j}\rightarrow u$ in $C^{1, \alpha}_{loc}(M^{\prime})\cap W^{1, p}_{loc}(M^{\prime})$.
It turns out that $u$ is a weak solution of problem \eqref{p-H}. The inequality $u\leq1$ and the uniqueness
of the solution can be directly obtained from the comparison theorem (see Lemma 2.2 in \cite{Ma--Yang--Yin--2025}).

\textbf{Step 3:} We will show
\begin{eqnarray*}
\mathrm{Cap}_{p}(\Sigma, M^{\prime})
= \int_{M^{\prime}}|\nabla u|^p.
\end{eqnarray*}
One the one hand, we know from Fatou's lemma
\begin{eqnarray}\label{Yin15}
\lim_{R\rightarrow +\infty}\mathrm{Cap}_{p}(\Sigma, M^{\prime}_{R})
=\mathrm{Cap}_{p}(\Sigma, M^{\prime})
=\lim_{R\rightarrow +\infty} \int_{M^{\prime}_{R}}|\nabla u_R|^p
\geq\int_{M^{\prime}}|\nabla u|^p.
\end{eqnarray}

On the other hand, since the sequence $\nabla u_R$ is uniformly bounded for large $R$
in $L^{p}(M^{\prime})$ and $u_{R_j}\rightarrow u$ in
$C^{1, \alpha}_{loc}(M^{\prime})$, we know $\nabla u_{R_j}\rightharpoonup \nabla u$ in
$L^{p}(M^{\prime})$. Then, applying Mazur lemma, we obtain that for any $\varepsilon>0$, there exists
$\lambda_i\geq 0$ $(i=1, 2, ..., N)$ such that
\begin{eqnarray*}
\parallel\nabla u-\sum_{i=1}^{N}\lambda_i\nabla u_{R_i}\parallel_{L^p}<\varepsilon.
\end{eqnarray*}
Using the fact $u_{R_i} \in \mathcal{A}_{R_i}$ and the inequality \eqref{Yin14}, we conclude
\begin{eqnarray}\label{Yin16}
\int_{M^{\prime}}|\nabla u|^p\geq \mathrm{Cap}_{p}(\Sigma, M^{\prime}).
\end{eqnarray}
Thus, the claim follows from \eqref{Yin15} and \eqref{Yin16}. This completes the proof.
\end{proof}

Finally, we obtain the follow asymptotic expansion of $u$. 
\begin{lemma}\label{lem:2.2}
Let $u$ be a solution to \eqref{p-H}. Then $u \in C^{1, \alpha}({M^{\prime}})$.
It is known that, $u$ has an asymptotic expansion
\begin{eqnarray*}
u=1- \frac{\mathfrak{c}_p}{a}r^{-a}+O_{2}(r^{-a-\tilde{\tau}}),
\quad \mbox{as} \quad r=|x|\rightarrow +\infty
\end{eqnarray*}
for $0<\tilde{\tau}<\min\{\tau, 1\}$, where
\begin{eqnarray*}
a:=\frac{3-p}{p-1}, \quad \mathfrak{c}_p:=\Big(\frac{\mathrm{Cap}_p}{2\pi}\Big)^{\frac{1}{p-1}}.
\end{eqnarray*}
\end{lemma}

\begin{remark}
Lemma A.2 in Mantoulidis--Miao--Tam \cite{Mantoulidis--Miao--Tam--2020} 
establishes the asymptotic behavior of harmonic functions on 
three-dimensional asymptotically flat manifolds with compact boundary. 
Later, Benatti--Fogagnolo--Mazzieri \cite{Benatti--Fogagnolo--Mazzieri--2024--MathAnn}, 
Theorem 3.1, proved the corresponding result for $p$-harmonic functions 
in the same geometric setting.

In contrast, here we consider $p$-harmonic functions on a 
three-dimensional asymptotically flat manifold with noncompact boundary. 
The main difference lies in the presence of a mixed boundary condition. 
To handle this additional difficulty, we construct suitable barrier functions
\[
w^{\pm}=1-C_{1}^{\pm}r^{-a}
\left(1\pm C_{2}r^{-\tilde{\tau}}
\left(1-\frac{x^{n}}{r}\right)\right),
\]
which allow us to control the asymptotic behavior of $u$.
\end{remark}

\begin{proof}
Applying the regularity results in \cite{DiBenedetto--1983--NATMA} and \cite[Theorems 1 and 2]{Lieberman--1988}, we know that
$u \in C^{1, \alpha}(M^{\prime}\cup \Sigma \cup \hat{\partial}M^{\prime})$.

We will construct two barrier functions to control the upper bound and lower bound of $u$ by using the weak comparison principle, so that we could get the expansion of $u$ at infinity. 

{\bf Step 1. Differences between the AF metric $g$ and flat metric $\bar g$. } 
Suppose that $g=\bar{g}+\sigma$, where $\sigma=O_{2}(|x|^{-\tau})$, $\tau>\frac{1}{2}$. 
And by calculation, we could get that 
\[
\begin{aligned}
    g^{ij}
    =& \bar{g}^{ij}-g^{ik}\sigma_{kl}\bar{g}^{lj} \\
    =& \bar{g}^{ij}-\sigma^{ij}+O(|\sigma|^{2}). \\
\end{aligned}
\]
\[
\begin{aligned}
    \det(g)
    =& \det(\bar{g}) \left[ 1+\operatorname{tr}_{\bar{g}}\sigma+O(|\sigma|^{2}) \right]. \\
\end{aligned}
\]
Define $\Gamma_{ij}^{k}$ such that 
\[
    \nabla_{\frac{\partial}{\partial x^{i}}}\frac{\partial}{\partial x^{j}}-\bar{\nabla}_{\frac{\partial}{\partial x^{i}}}\frac{\partial}{\partial x^{j}}=\Gamma_{ij}^{k}\frac{\partial}{\partial x^{k}}.
\]
Here $\Gamma_{ij}^{k}$ is not the usual Christoffel symbol. It's the difference between Levi-Civita connections with respect to two different metrics. Then, it is easy to see that
\[
\Gamma_{ij}^{k}=\frac{1}{2}\bar{g}^{kl}\left[ \nabla_{i}\sigma_{jk}+\nabla_{j}\sigma_{ik}-\nabla_{k}\sigma_{ij} \right].
\]
For any smooth function $f$,  we have the following calculations
\[
\begin{aligned}
    f_{ij}
    =& \nabla^{2}f\Big(\frac{\partial}{\partial x^{i}},\frac{\partial}{\partial x^{j}}\Big) 
    = \frac{\partial^2f}{\partial x^{i}\partial x^{j}}-\nabla_{\frac{\partial}{\partial x^{i}}}\frac{\partial}{\partial x^{j}}\Big(f\Big) \\
    =& \overline{\nabla}^{2}f\Big(\frac{\partial}{\partial x^{i}},\frac{\partial}{\partial x^{j}}\Big)-\Gamma_{ij}^{k}\frac{\partial}{\partial x^{k}}f 
    = (\overline{\nabla}^{2}f)_{ij}-\Gamma_{ij}^{k}f_{k}. \\
\end{aligned}
\]
\[
\begin{aligned}
    \Delta f-\bar{\Delta}f
    =& -f_{ij}\sigma^{ij}
    -g^{ij}\Gamma_{ij}^{k}f_{k}
    +O(|\bar{\nabla}^{2}f||\sigma^{2}|)+O(|\bar{\nabla}f||\sigma||\nabla\sigma|) \\
    =& -\langle \bar{\nabla}^{2}f,\sigma \rangle
    -\langle \operatorname{div}_{g}\sigma-\frac{1}{2}d\operatorname{tr}_{g}\sigma, df \rangle
    +O(|\bar{\nabla}^{2}f||\sigma^{2}|)+O(|\bar{\nabla}f||\sigma||\nabla\sigma|). \\
\end{aligned}
\]
Besides,
\[
\begin{aligned}
    |\nabla f|^{\alpha}
%    =& (|\nabla f|^{2})^{\frac{\alpha}{2}} \\
%    =& (g^{ij}f_{i}f_{j})^{\frac{\alpha}{2}} \\
%    =& \Big( (\bar{g}^{ij}-\sigma^{ij}+O(|\sigma|^{2}) )f_{i}f_{j} \Big)^{\frac{\alpha}{2}} \\
%    =& \Big( |\bar{\nabla}f|_{\bar{g}}^{2}-\sigma^{ij}f_{i}f_{j}+O(|\sigma|^{2}|\bar{\nabla}f|_{\bar{g}}^{2}) \Big)^{\frac{\alpha}{2}} \\
    =& |\bar{\nabla}f|_{\bar{g}}^{\alpha} 
    \Bigg( 
    1-\frac{\alpha}{2}\sigma^{ij}\frac{f_{i}}{|\bar{\nabla}f|_{\bar{g}}}\frac{f_{j}}{|\bar{\nabla}f|_{\bar{g}}}+O(|\sigma|^{2}) 
    \Bigg). \\
\end{aligned}
\]
\[
\begin{aligned}
    \nabla^{2}f(\nabla f, \nabla f)
%    =& f_{ij}f^{i}f^{j} \\
%    =& \Big( (\overline{\nabla}^{2}f)_{ij}-\Gamma_{ij}^{k}f_{k} \Big)f^{i}f^{j} \\
%    =& \Big( (\overline{\nabla}^{2}f)_{ij}-\Gamma_{ij}^{k}f_{k} \Big) g^{is}f_{s}g^{jm}f_{m} \\
%    =& \Big( (\overline{\nabla}^{2}f)_{ij}-\Gamma_{ij}^{k}f_{k} \Big) (\bar{g}^{is}-\sigma^{is}+O(r^{-2\tau}))f_{s}(\bar{g}^{jm}-\sigma^{jm}+O(r^{-2\tau}))f_{m} \\
    =& \overline{\nabla}^{2}f(\bar{\nabla}f,\bar{\nabla}f)-\bar{g}^{is}\bar{g}^{jm}\Gamma_{ij}^{k}f_{k}f_{s}f_{m}-2(\overline{\nabla}^{2}f)_{ij}\bar{g}^{is}\sigma^{jm}f_{s}f_{m} \\
    & +O(|\overline{\nabla}^{2}f||\bar{\nabla}f|^{2}|\sigma|^{2})
    +O(|\bar{\nabla}f|^{3}|\sigma||\nabla \sigma|). \\
\end{aligned}
\]

{\bf Step 2.  Barrier function and its derivative. }
In the coordinate chart at infinity, $\{ x^{1}, x^{2}, \ldots, x^{n} \}$.
We will use the following convention: tangential indices 
$i, j, k$ take values $1, \ldots, n-1$, while full indices 
$\alpha, \beta, \gamma$ range over $1, \ldots, n$.

Recall that on $\hat{\partial}M'=M'\cap\{x|x^{n}=0\}$, 
the outward unit normal vector field $\mu$ satisfies that
\[
\left\{
\begin{aligned}
    g(\mu,\frac{\partial}{\partial x^{i}})=& 0, \quad i=1,\ldots,n-1 \\
    g(\mu,\mu)=& 1, \\
    g(\mu,\frac{\partial}{\partial x^{n}})<&0. \\
\end{aligned}
\right.
\]
Thus
\[
\mu=-(g^{nn})^{-\frac{1}{2}}g^{n\alpha}\frac{\partial}{\partial x^{\alpha}}
=-(1-\frac{1}{2}\sigma^{nn}+O(|\sigma|^{2}))\frac{\partial}{\partial x^{n}}+(\sigma^{ni}+O(|\sigma|^{2}))\frac{\partial}{\partial x^{i}}.
\]

Then we consider the barrier function 
\[
w^{\pm}=1-C_{1}^{\pm}r^{-a}(1\pm C_{2}r^{-\tilde{\tau}}(1-\frac{x^{n}}{r})),
\]
here $a=\frac{n-p}{p-1}$ and $0<\tilde{\tau}<\min\{\tau, 1\}$. Besides, $r_{1}$ is a large constant, 
$C_{1}^{-}=r_{1}^{a}(1-\min\limits_{x\in S_{r_{1}}}u(x))$, $C_{1}^{+}=r_{1}^{a}(1-\max\limits_{x\in S_{r_{1}}}u(x))$,
$C_2$ is large enough such that when $r\ge r_{1}$, $C_{2}r^{-\tilde{\tau}}\ge a|\sigma|$. (Notice that $\sigma=O_{2}(r^{-\tau})$.)

Firstly, we calculate the normal derivative along $\mu$ on $\hat{\partial}M'$. 
\[
\begin{aligned}
    \mu(w^{\pm})
    =& -(1-\frac{1}{2}\sigma^{nn}+O(|\sigma|^{2}))\frac{\partial w^{\pm}}{\partial x^{n}}+(\sigma^{ni}+O(|\sigma|^{2}))\frac{\partial w^{\pm}}{\partial x^{i}} \\
    =& -C_{1}^{\pm}
    \Big( - (\pm C_{2})
    r^{-a-\tilde{\tau}-1}-ar^{-a-1}\sigma^{ni}\frac{x^{i}}{r}+O(r^{-a-\tilde{\tau}-\tau-1})
    \Big)
\end{aligned}
\]
i.e. $\mu(w^{+})\ge0$, $\mu(w^{-})\le 0$.

Next, the first derivative of $w^{\pm}$ is

% Gradient (first derivatives)
\[
\begin{aligned}
    \partial_i w^{\pm}
    =& C_{1}^{\pm} a r^{-a-1} \frac{x^{i}}{r} 
    + C_{1}^{\pm}(\pm C_{2}) r^{-a-\tilde{\tau}-1}
    \Big(
    \delta_{i n} 
    + (a+\tilde{\tau}) \frac{x^{i}}{r}
    - (a+\tilde{\tau}+1) \frac{x^{i}}{r}\frac{x^{n}}{r}
    \Big). \\
\end{aligned}
\]

\[
|\bar{\nabla} w^{\pm}|_{\bar{g}}
=aC_{1}^{\pm}r^{-a-1}
\Big(
1+(1+\frac{\tilde\tau}{a})(\pm C_{2})r^{-\tilde{\tau}}(1-\frac{x^{n}}{r})+O(r^{-2\tilde{\tau}})
\Big).
\]

Then, the second derivative of $w^{\pm}$ is 
\[
\begin{aligned}
    \partial_i\partial_j w^{\pm}
    =& C_{1}^{\pm} a r^{-a-2}
    \Big( \delta_{ij} - (a+2) \frac{x^{i}}{r}\frac{x^{j}}{r} \Big) \\
    & + C_{1}^{\pm}(\pm C_{2})(a+\tilde{\tau}) r^{-a-\tilde{\tau}-2}
    \big( \delta_{ij}
    - (a+\tilde{\tau}+2)\frac{x^{i}}{r}\frac{x^{j}}{r} \big) \\
    & - C_{1}^{\pm}(\pm C_{2})(a+\tilde{\tau}+1) r^{-a-\tilde{\tau}-2}
    \big( \delta_{jn} \frac{x^{i}}{r}
    + \delta_{in} \frac{x^{j}}{r}
    + \delta_{ij} \frac{x^{n}}{r} \big) \\
    & + C_{1}^{\pm}(\pm C_{2}) (a+\tilde{\tau}+1)(a+\tilde{\tau}+3) r^{-a-\tilde{\tau}-2}\frac{x^{i}}{r}\frac{x^{j}}{r}\frac{x^{n}}{r}. \\
\end{aligned}
\]
By taking trace,  we have
\[
\begin{aligned}
    \overline{\Delta}w^{\pm}
    =& C_{1}^{\pm} a r^{-a-2}
    \Big( n-a-2 \Big) \\
    & + C_{1}^{\pm}(\pm C_{2})(a+\tilde{\tau}) r^{-a-\tilde{\tau}-2}
    \big( n-a-2-\tilde{\tau} \big) \\
    & - C_{1}^{\pm}(\pm C_{2})(a+\tilde{\tau}+1) r^{-a-\tilde{\tau}-2}
    \big( n-a-1-\tilde{\tau} \big) \frac{x^{n}}{r}. \\
\end{aligned}
\]
Besides, we could get that
\[
\begin{aligned}
    \bar{\nabla}^{2}w^{\pm}(\bar{\nabla}w^{\pm},\bar{\nabla}w^{\pm})
    =& \partial_i\partial_j w^{\pm}\partial_i w^{\pm}\partial_j w^{\pm} \\
    &= 
    -\,(C_{1}^{\pm})^{3}\,a^{3}(a+1)\; r^{-3a-4}
    \\[6pt]
    &\quad - (\pm C_{2})\,(C_{1}^{\pm})^{3}\,a^{2}(a+\tilde{\tau})(3a+3+\tilde{\tau})\; r^{-3a-\tilde{\tau}-4}\,
    \Big( 1- \frac{x^{n}}{r} \Big)
    \\[6pt]
    &\quad + O\!\big((C_{1}^{\pm})^{3}C_{2}^{2}r^{-3a-2\tilde{\tau}-4}\big). \\
\end{aligned}
\]
Thus 
\[
\begin{aligned}
    \bar{\Delta}_{p}w^{\pm}
    =& |\bar{\nabla} w^{\pm}|_{\bar{g}}^{p-2}\bar{\Delta}w^{\pm}
    +(p-2)|\bar{\nabla} w^{\pm}|_{\bar{g}}^{p-4}
    \bar{\nabla}^{2}w^{\pm}(\bar{\nabla}w^{\pm},\bar{\nabla}w^{\pm}) \\
    =& - (C_{1}^{\pm})^{p-1}(\pm C_{2})a^{p-2}r^{-(a+1)(p-1)-1} \times \\
    & \times 
    \Bigg(
    (a+\tilde{\tau})(p-1)\tilde{\tau}(1-\frac{x^{n}}{r})r^{-\tilde{\tau}}+(n-1)\frac{x^{n}}{r}r^{-\tilde{\tau}}+O(r^{-2\tilde{\tau}})
    \Bigg). \\
\end{aligned}
\]

{\bf Step 3. Error due to the asymptotic flatness. }
 We consider the difference between $\Delta_{p}w$ and $\overline{\Delta}_{p}w$.
Notice that
\begin{align*}
    \Delta_{p}w
    =& |\nabla w|^{p-2}\Delta w+(p-2)|\nabla w|^{p-2}\nabla^{2}w
    \Big( \frac{\nabla w}{|\nabla w|},\frac{\nabla w}{|\nabla w|} \Big) \\
    % =& \overline{\Delta}_{p}w
    % +|\nabla w|^{p-2}( \Delta w-\overline{\Delta}w )
    % +(|\nabla w|^{p-2}-|\bar{\nabla} w|^{p-2}) \overline{\Delta}w \\
    % & +(p-2)|\nabla w|^{p-4} \Big( \nabla^{2}w (\nabla w, \nabla w) -\bar{\nabla}^{2}w (\bar{\nabla} w, \bar{\nabla} w)\Big) \\
    % & +(p-2)(|\nabla w|^{p-4}-|\bar{\nabla} w|^{p-4})\bar{\nabla}^{2}w (\bar{\nabla} w, \bar{\nabla} w). \\
\end{align*}
It is easy to see that
\begin{align*}
    \Delta_{p}w
    - \overline{\Delta}_{p}w
    =& |\bar{\nabla} w|_{\bar{g}}^{p-2}( \Delta w-\overline{\Delta}w ) 
     +(|\nabla w|^{p-2}-|\bar{\nabla} w|_{\bar{g}}^{p-2}) \overline{\Delta}w \\
    & +(p-2)|\bar{\nabla} w|_{\bar{g}}^{p-4} \Big( \nabla^{2}w (\nabla w, \nabla w) -\bar{\nabla}^{2}w (\bar{\nabla} w, \bar{\nabla} w)\Big) \\
    & +(p-2)(|\nabla w|^{p-4}-|\bar{\nabla} w|_{\bar{g}}^{p-4})\bar{\nabla}^{2}w (\bar{\nabla} w, \bar{\nabla} w) \\
    & +(|\nabla w|^{p-2}-|\bar{\nabla} w|_{\bar{g}}^{p-2})( \Delta w-\overline{\Delta}w ) \\
    & +(p-2)(|\nabla w|^{p-4}-|\bar{\nabla} w|_{\bar{g}}^{p-4})\Big( \nabla^{2}w (\nabla w, \nabla w) -\bar{\nabla}^{2}w (\bar{\nabla} w, \bar{\nabla} w)\Big).
\end{align*}

Thus when $r$ large enough,  we can see that
\[
\begin{aligned}
    |\Delta_{p}w^{\pm}-\bar{\Delta}_{p}w^{\pm}|
    & \le (C_{1}^{\pm})^{p-1}a^{p-2}r^{-(a+1)(p-1)}  
     \times \Bigg( C(n,p,\tilde{\tau})
    \Big(
    r^{-1}|\sigma|+|\nabla\sigma|
    \Big) 
    %+O(r^{-\tau-\tilde{\tau}})
    \Bigg), \\
\end{aligned}
\]
where $C(n,p,\tilde{\tau})$ is a constant with respect to $n$, $p$ and $\tilde{\tau}$, independent of $C_{1}^{\pm}$ and $C_{2}$.

By taking $C_{2}$ large enough, we can get that
\[
\Delta_{p}w^{+}\le \bar{\Delta}_{p}w^{+}+|\Delta_{p}w^{+}-\bar{\Delta}_{p}w^{+}|\le0,
\]
and
\[
\Delta_{p}w^{-}\ge \bar{\Delta}_{p}w^{-}-|\Delta_{p}w^{-}-\bar{\Delta}_{p}w^{-}|\ge0.
\]

{\bf Step 4. Application of the weak comparison. }
Define that $\phi=\max\{u-w^{+},0\}$. By exploiting the asymptotic behavior derived previously, the following integrals can be shown to be integrable.
\[
\begin{aligned}
    0& \le \int_{\{ \phi>0 \}}\phi(\Delta_{p}u-\Delta_{p}w^{+})
    +g\Bigg( |\nabla u|^{p-2}\nabla u-|\nabla w^{+}|^{p-2}\nabla w^{+}, \nabla \phi \Bigg) \\
    =& 
    \int_{\{ \phi>0 \}}\operatorname{div}(\phi|\nabla u|^{p-2}\nabla u-\phi|\nabla w^{+}|^{p-2}\nabla w^{+}) \\
    =& \int_{\{ \phi>0 \}\cap S_{r_{1}}} \phi \cdot g\Bigg( |\nabla u|^{p-2}\nabla u-|\nabla w^{+}|^{p-2}\nabla w^{+}, -\nu \Bigg) \\
    &+ \int_{\{ \phi>0 \}\cap \hat{\partial}M'} \phi \cdot g\Bigg( |\nabla u|^{p-2}\nabla u-|\nabla w^{+}|^{p-2}\nabla w^{+}, \mu \Bigg) \\
    =& \int_{\{ \phi>0 \}\cap \hat{\partial}M'} -\phi|\nabla w^{+}|^{p-2} \mu(w^{+}) \\
    \le& 0. \\
\end{aligned}
\]
Thus, we have $\phi \equiv 0$, which implies $u \le w^+$. 
Similarly, one can show that $w^- \le u$. 
Consequently, the \(C^{0}\) asymptotic expansion of \(u\) follows directly from the definition of the \(p\)-capacity. The corresponding \(C^{2}\) estimates can be established by a similar argument to that in \cite{Chrusciel--1990}.
\end{proof}

{
We next establish the connectedness of the regular level sets, following the approach of Huisken–Ilmanen (\cite{Huisken--Ilmanen--2001}, Lemma $4.2$) and Koerber (\cite{Koerber--2023}, Lemma $3.7$). This result will be useful in our application of the Gauss–Bonnet formula.
\begin{lemma}\label{lem:2.3}
    Under the assumption of Theorem \ref{Thm:1.02}. Assume that $M'$ is simply connected and $\Sigma$ is connected.
    Let $u$ be a solution to \eqref{p-H}. Then there is no strict local maxima or minima in the interior of $M'$. Besides, the level set $\Sigma_{t}$ of $u$ is a connected free-boundary hypersurface. 
\end{lemma}
\begin{proof}
    \textbf{Step 1. There are no strict local maxima or minima.}
    
    Suppose that $u$ attains a strict local maximum $u_{\mathrm{slmax}}$ on a connected, precompact set $E \subset\subset M'$. Then for any regular value $t_{\max}$ of $u$ with $t_{\max} < u_{\mathrm{slmax}}$, we have $E \subset \{u \ge t_{\max}\}$. Define
    \[
    v =
    \begin{cases}
    t_{\max}, & \text{on } \{u \ge t_{\max}\},\\[2pt]
    u, & \text{on } \{u < t_{\max}\}.
    \end{cases}
    \]
    Then $v \in \mathcal{A}$, and
    \[
    0 = \int_{\{u \ge t_{\max}\}} |\nabla v|^{p}
    \;\ge\; \int_{\{u \ge t_{\max}\}} |\nabla u|^{p}.
    \]
    Hence $|\nabla u| = 0$ a.e.\ on $\{u \ge t_{\max}\}$, so $u = t_{\max}$ on $E$, a contradiction. The same argument shows that $u$ has no strict local minimum.
    
    \textbf{Step 2. The set $\{u > t\}$ is connected.}
    
    Suppose, to the contrary, that $\{u > t\}$ is not connected. Since $M'$ has only one end, there exists a connected component $E'$ of $\{u > t\}$ that does not contain the end at infinity. Because $u$ has no strict local maximum, $E'$ must intersect the boundary $\widehat{\partial} M'$.
    
    However, since $\langle \nabla u, \mu \rangle = 0$ on $\widehat{\partial} M'$, the function $u$ cannot attain a strict local maximum on $\widehat{\partial} M'$; otherwise this would contradict Hopf's Lemma. Therefore $E'$ is a connected precompact region with boundary, and $u$ must attain a strict local maximum inside $E'$. By Step~1, this is impossible.
    
    \textbf{Step 3. The set $\{u < t\}$ is connected.}
    
    Suppose that $\{u < t\}$ is not connected. Then there exists a connected component $E''$ that does not contain $\Sigma$. As in Step~2, $E''$ is a connected, precompact region with boundary, and $u$ attains a strict local minimum in $E''$, which is impossible by Step~1.
    
    \textbf{Step 4. The regular level set $\{u = t\}$ is connected.}
    
    Suppose instead that $\{u = t\}$ is not connected, and let $E_{1}$ and $E_{2}$ be two distinct connected components. Choose three points:
    \[
    p_{0} \in \Sigma,\qquad p_{1} \in E_{1}, \qquad p_{2} \in E_{2}.
    \]
    By the connectedness of $\{u < t\}$, $\{u > t\}$, $E_{1}$, and $E_{2}$, we can find curves
    \[
    \gamma_{01} \subset \{u < t\} \text{ joining } p_{0} \text{ to } p_{1}, \qquad
    \gamma_{12} \subset \{u > t\} \text{ joining } p_{1} \text{ to } p_{2},
    \]
    and
    \[
    \gamma_{20} \subset \{u < t\} \text{ joining } p_{2} \text{ to } p_{0}.
    \]
    These three curves form a closed loop based at $p_{0}$ that intersects $E_{1}$ exactly once. By intersection theory, such a loop cannot be homotoped to a point, contradicting the simple connectedness of $M'$. Hence the regular level set $\{u = t\}$ must be connected.
    
    \textbf{Step 5. The regular level set $\{u = t\}$ is a free–boundary hypersurface.}
    
    Suppose not. Then
    \[
    \{u = t\} \cap \widehat{\partial} M' = \emptyset.
    \]
    But this is impossible, since $u$ is continuous and $\widehat{\partial} M'$ is connected. Hence the regular level set $\{u = t\}$ must intersect $\widehat{\partial} M'$, and therefore it is a free–boundary hypersurface, since $\langle \nabla u, \mu \rangle = 0$ on $\widehat{\partial} M'$.

\end{proof}

\begin{remark}
    If we replace the simply connected assumption by $b_{1}(M') = 0$, then together with $H > 0$, we may be able to prove that $M'$ is simply connected by referring to Oronzio \cite{Oronzio--2025} and Koerber \cite{Koerber--2023}.
\end{remark}

}
%%%%%%%%%%%%%%%%%%%%%%%%%%%%
\section{$p$-capacity of Schwarzschild half space}\label{Sec:03}
		
Given $m\in\mathbb{R}$ and {$r_0\geq |m|$. Let  $(\mathcal{M}_{m, r_0}^{+,3},   g_{m})$ be the spatial Schwarzschild half space  of mass $m$ outside a rotationally symmetric half sphere, given by \eqref{schwarzschild}. Here $m$ can be taken to be a negative real number, we still call it the Schwarzschild half space.  Now we will follow Xia-Yin-Zhou's idea to calculate the $p$-capacitary function in $(\mathcal{M}_{m, r_0}^{+,3},   g_{m})$ (for detail see \cite{Xia--Yin--Zhou--2024}).
		\begin{proposition} \label{Appendix: basic_solution}
			Let   $(\mathcal{M}_{m, r_0}^{3,+},   g_{m})$ be the Schwarzschild half space of mass $m$  outside a rotationally symmetric half-sphere given by \eqref{schwarzschild} and $u$ be the solution to \eqref{p-H}. Then  $u(x)= f_{m, r_0}(r), r=|x|$, where $f_{m, r_0}$ is given by
			\begin{eqnarray*}
				f_{m, r_0}(r)&:=& 1 -\int_{r}^{\infty} \mathfrak{c}_{p} s^{-a-1}(1+\frac{m}{s})^{-2a}ds
				=1-\frac{I_{a}(\frac{m}{r})}{I_{a}(\frac{m}{r_{0}})},
			\end{eqnarray*}
			where $\mathfrak{c}_{p}$ and $m$ can be related by
			\begin{equation} \label{equ:m}
				|m| =  \left( \mathfrak{c}_{p}I_{a}(\frac{m}{r_{0}}) \right)^{\frac{1}{a}}.
			\end{equation}
		\end{proposition}
		%%%%%%
		\begin{proof}
			Since $(\mathcal{M}_{m, r_0}^{3,+},   g_{m})$ is rotationally symmetric and the solution to  \eqref{p-H} is unique, $u$ is rotationally symmetric, i.e. there exists some $f(r)$ such that $u(x)=f(r)$.
			The Euclidean $p$-Lapacian on $u$ gives
			$$
			\bar{\Delta}_{p} u
			= \left( f'(r) \right)^{p - 2} \left[ (p - 1) f''(r) + 2 \frac{1}{r}f'(r) \right].
			$$
			Let $w = 1 + \frac{m}{ r}$ so that $ g=w^4\bar g$. By using the transformation formula for $p$-Laplacian under conformal change,
			we get that
			\begin{eqnarray*}
				\Delta_{p} u&=& w^{- 2p} \left( \bar{\Delta}_{p} u + 2(3 - p) |\bar{\nabla} u|^{p - 2}_{\bar{g}} \<\bar{\nabla}u, \bar\nabla\ln w\> \right)
				\\&=& w^{- 2p} (f'(r))^{p - 2} \left[ (p - 1) f''(r) + 2 \frac{1}{r}f'(r) + 2(3 - p) f'(r) \partial_{r}(\operatorname{ln} w) \right].
			\end{eqnarray*}
			It follows from $\Delta_{p} u = 0$ that
			$$f'(r) = C r^{-a-1}\left(1+\frac{m}{r}\right)^{-2a},$$ where $C$ is a positive constant. We claim that $C=\mathfrak{c}_{p}$. Indeed,
			since $$|\nabla u|_{g} =  w^{-2} |\bar{\nabla} u|_{\bar{g}} = w^{-2} f'(r),$$ we see
			$$
			\operatorname{Area}_{g} (\{r=r_0\})
			= \int_{\{ r = r_{0} \}}  d \sigma
			= \int_{\{ r = r_{0} \}} w^{4} d \bar{\sigma}
			= (w(r_{0}))^{4} 2 \pi r_{0}^{2}.
			$$
			%Here $d \bar{\sigma}$ is the Euclidean volume form.
			It follows that
			\begin{eqnarray*}
				\mathfrak{c}_{p}= \left( \frac{1}{2\pi} \int_{\partial \mathcal{M}_{m, r_0}} |\nabla u|_{g}^{p-1} d \sigma \right)^{\frac{1}{p-1}}
				%= \left( \frac{1}{4 \pi} \int_{\{ r = r_{0} \}} \left( w^{-2} f'(r) \right)^{p-1} w^{4} d \bar{\sigma} \right)^{\frac{1}{p-1}}
				= (w(r_{0}))^{ 2\frac{3-p}{p-1}} r_{0}^{\frac{2}{p-1}} f'(r_{0})
				= C.
			\end{eqnarray*}

			Since $u \rightarrow 1$ as $r \rightarrow \infty$, we get
			\begin{equation*}
				f(r) = 1 - \int_{r}^{+ \infty} \mathfrak{c}_{p} t^{-a-1}\left(1+ \frac{m}{t}\right)^{-2a} d t.
			\end{equation*}
			By a change of variable $s=\frac{|m|}{t}$, we see that
			\begin{equation*}
				f(r) = 1 - |m|^{-a} \mathfrak{c}_{p}\int_{0}^{\frac{|m|}{r}}  s^{a-1}(1+\operatorname{sgn}(m)s)^{-2a} d s=1 - |m|^{-a} \mathfrak{c}_{p}I_{a}(\frac{m}{r}).
			\end{equation*}
			Since $u = 0$ on boundary $\{r=r_{0}\}$, we get
			$$
			|m|^{-a} \mathfrak{c}_{p}I_{a}(\frac{m}{r_0})=1.
			$$
			
		\end{proof}

		\begin{proposition}[Proposition 2.2 \cite{Xia--Yin--Zhou--2024}]
			Let   $(\mathcal{M}_{m, r_0}^{3,+},   g)$ be a spatial Schwarzschild half space outside a rotationally symmetric half sphere of mass $m$ given by \eqref{schwarzschild} and $u$ be the solution to \eqref{p-H}. Let $H$ be the mean curvature of the level set $S_{+,r}^2=\{x\in\mathbb R^3_+|\, |x|=r\}$. Then
			\begin{eqnarray*}
				H=\frac{2}{r}\left(1 + \frac{m}{ r}\right)^{-3}\left(1- \frac{m}{r}\right), 	\end{eqnarray*}
			\begin{equation}
				\frac{H}{2|\nabla u|_{g}}=\eta(r)= \mathfrak{c}_{p}^{-1}r^{a}\left(1 + \frac{m}{ r}\right)^{2a-1}\left((1- \frac{m}{r}\right).
			\end{equation}
		\end{proposition}
		%\begin{proof}
			%Denote $\bar H$ and $\bar \nu$ be the mean curvature and the unit normal of $S_r=\{|x|=r\}$ under the Euclidean metric $\bar g$, respectively.
			%By the transformation formula for mean curvature under conformal change, we have	$$
			%H = w^{ - 3 }\left( \bar{H} w + 4 \frac{\partial w}{\partial \bar{\nu}} \right) = w^{ - 3 } \left(  \frac{2}{r} w + 4 \frac{\partial w}{\partial r} \right)  = \frac{2}{r}\left(1 + \frac{m}{ r}\right)^{-3}\left(1- \frac{m}{r}\right).
		%	$$
		%	Note that $|\nabla u|_{g} = w^{-2} f'(r),$ the second assertion follows.
%\end{proof}

\section{Monotonicity of $F(t)$}\label{Sec:04}
		
		The aim of this section is to prove the monotonicity of $F(t)$ in Theorem \ref{thm:Monot}.	
		\subsection{Monotonicity of $F(t)$ when $|\nabla u|\neq 0$}
		We assume that $|\nabla u|\neq 0$ in $M'$, then $u\in C^\infty(M')$.
		Consider the level set $\Sigma_t=\{u=f(t)\}$, where $f$ is a given one-variable function. One sees readily that $\{\Sigma_t\}$ satisfies the flow equation
		\begin{equation} \label{meancurv-evol}
			\begin{cases}
				\Psi: \Sigma \times (r_0, + \infty) \rightarrow \ M', \\
				\partial_t \Psi(p, t)=f'(t)\frac{\nabla u}{|\nabla u|^2}=f'(t)\frac{1}{|\nabla u|}\nu \ \ \ {\rm on}\ \Sigma\times(r_0, + \infty)\\
\langle \mu\circ\Psi,\nu\circ\Psi\rangle=0\ \ \ {\rm on}\ \partial\Sigma\times(r_0, + \infty),
			\end{cases}
		\end{equation}
where $\mu$ and $\nu$ is outward unit normal to $\hat\partial M'=\partial M'\setminus\Sigma$ and $\Sigma_t$, respectively,

		The following basic facts are well-known, see for example \cite{Hirsch--Miao--Tam--2024}.
		\begin{lemma}
			The mean curvature of a regular level set $\Sigma_t$ is given by
			\begin{eqnarray}
				&&H= (1-p)\frac{1}{ |\nabla u|}u_{\nu\nu}=\frac{1-p}{2}\frac{1}{|\nabla u|^{2}}\nu(|\nabla u|^2), \label{equ: H}
				%&&\operatorname{div}_{\Sigma_t} \left( \frac{\nabla u}{|\nabla u|^2} \right)\label{equ: div_nu}
				%= H |\nabla u|^{- 1}.
			\end{eqnarray}	
			where $u_{\nu\nu}=\nabla^2 u \left( \nu, \nu \right)$.
		\end{lemma}
		%By the standard calculation, we have the first variation formula.
		%%%%%%
		The evolution equation for level set flow \eqref{meancurv-evol} is as following, see \cite{Marquardt--2017}.
		\begin{lemma}\label{lem:4.2}
			\begin{equation} \label{E1}
				\frac{\p}{\p t} H
				= - f'(t) \left( \Delta_{\Sigma_t} \left( \frac{1}{|\nabla u|} \right)
				+ \left( |h|^2 + \operatorname{Ric}_{M'} (\nu, \nu) \right)\frac{1}{|\nabla u|} \right).
			\end{equation}	
\begin{equation}\label{E2}
\frac{\nabla_{\mu} |\nabla u|}{|\nabla u|^2}=-\frac{h_{\nu\nu}^{\hat\partial M'}}{|\nabla u|},
\end{equation}
where $h_{\nu\nu}^{\hat\partial M'}$ is second fundamental form of $\hat\partial M'$.
			%where $h$ is the second fundamental form and $\operatorname{Ric}_M$ is the Ricci curvature of $(M, g)$.
		\end{lemma}

Next we prove the following variational formula.
		\begin{lemma}\label{lem:2.4} Along the level set flow, we have
			\begin{equation*}\label{E3}
\begin{aligned}
				\frac{d}{d t}  \int_{\Sigma_t} |\nabla u|^2
				= - a f'(t)  \int_{\Sigma_t} H |\nabla u|.
\end{aligned}
\end{equation*}
%%%%%%%%
\begin{equation*}\label{E4}
\begin{aligned}
				\frac{d}{d t} \int_{\Sigma_t} H |\nabla u|
				=& -f'(t) \Bigg\{  \int_{\Sigma_t} |\nabla u|^{-2}|\nabla_{\Sigma_t} |\nabla u||^2 + \frac{1}{2}( R_{M'} - K_{\Sigma_t}+ |\overset{\circ}{h}|^2)\\
& + \int_{\Sigma_t}\frac{2 a + 1}{4}H^2+ \int_{\partial\Sigma_t}h^{\hat\partial M'}_{\nu\nu}\Bigg\},
\end{aligned}
\end{equation*}
			where $\overset{\circ}{h}$ denotes the traceless part of the second fundamental form of the level set.
		\end{lemma}
		%%%%%%
		%%%%%%
		\begin{proof}
			Recall that the variation field of the level set flow is $\partial_t \Psi = f'(t) \frac{1}{|\nabla u|} \nu $.
		
			Using \eqref{equ: H}, one computes
			\begin{align*}
				\frac{d}{d t}  \int_{\Sigma_t} |\nabla u|^2
				=& \int_{\Sigma_t}  f'(t) \frac{1}{|\nabla u|} \nu (|\nabla u|^2)+ |\nabla u|^2 H f'(t)  \frac{1}{|\nabla u|} \\
				=& \int_{\Sigma_t}  f'(t) \frac{2}{1 - p} |\nabla u| H + f'(t) H |\nabla u|  \\
				=& - af'(t)  \int_{\Sigma_t} H |\nabla u|.
			\end{align*}
			On the one hand, by the divergence theorem and \eqref{E2}, we obtain
			\begin{align}\label{equ: div}
				\int_{\Sigma_t} |\nabla u| \Delta_{\Sigma_t} \left( \frac{1}{|\nabla u|} \right)
				=& \int_{\Sigma_t} |\nabla u|^{-2} \left| \nabla_{\Sigma_t} {|\nabla u|} \right|^{2} +\int_{\partial\Sigma_t}-|\nabla u|^{-1}\langle\nabla |\nabla u|,\mu\rangle\\
=& \int_{\Sigma_t} |\nabla u|^{-2} \left| \nabla_{\Sigma_t} {|\nabla u|} \right|^{2} +\int_{\partial\Sigma_t}h^{\hat\partial M'}_{\nu\nu}.
			\end{align}
			On the other hand, the Gauss equation tells that
			\begin{equation}\label{equ:gauss}
				2 \operatorname{Ric}_M (\nu, \nu) = R_M-2K_{\Sigma_t}+\frac12H^2-|\overset\circ{h}|^2.
			\end{equation}
			It follows from \eqref{equ: H}, \eqref{E1}, \eqref{equ: div} and \eqref{equ:gauss} that
			\begin{align*}
				\frac{d}{d t} \left(  \int_{\Sigma_t} H |\nabla u|    \right)
				%=& \int_{\Sigma_t} \frac{d}{d t} \left( H |\nabla u| \right) + H |\nabla u| f'(t)  H |\nabla u|^{- 1}   \\
				%=& \int_{\Sigma_t} \frac{d}{d t} \left( H \right) |\nabla u| + H X \left(|\nabla u| \right) + H |\nabla u| f'(t)  H |\nabla u|^{- 1}   \\
				=& \int_{\Sigma_t} - f'(t) \left( \Delta_{\Sigma_t} \left( \frac{1}{|\nabla u|} \right) + \left( |h|^2 + \operatorname{Ric}_M (\nu, \nu)  \right)\frac{1}{|\nabla u|} \right) |\nabla u| \\
				&+ \int_{\Sigma_t} H f'(t) \frac{1}{|\nabla u|} \nu (|\nabla u|) + H|\nabla u| Hf'(t)   |\nabla u|^{- 1} \\
				=& -f'(t) \Bigg\{  \int_{\Sigma_t} |\nabla u|^{-2}|\nabla_{\Sigma_t} |\nabla u||^2 + \frac{1}{2}( R_{M'} - 2K_{\Sigma_t}+ |\overset{\circ}{h}|^2)\\
&+\int_{\Sigma_t} \frac{2 a + 1}{4}H^2 +\int_{\partial\Sigma_t}h^{\hat\partial M'}_{\nu\nu}\Bigg\}.
			\end{align*}
		\end{proof}
%%%%%%%%%%%%%%%%%%%%%%%%%%%%%%%%%%%%%%%%%%%%%%%%%%%%%%%%%%%%%%%%%%%
\begin{proposition}
			Under assumptions of Theorem \ref{thm:Monot}.  Assume in addition that $|\nabla u|\neq 0$ in $M'$. Then $F'(t)\le 0$.
		\end{proposition}
		\begin{proof}
			Since $|\nabla u| \neq 0$, it  follows from Lemma \ref{lem:2.4} that
			\begin{align*}
				F'(t)=& 2\pi \gamma'(t) + \alpha'(t) \int_{\Sigma_t} H |\nabla u|   + \beta'(t) \int_{\Sigma_t} |\nabla u|^2   \\
				&- \alpha(t) f'(t) \left\{  \int_{\Sigma_t} |\nabla u|^{-2}|\nabla_{\Sigma_t} |\nabla u||^2 + \frac{1}{2}( R_{M} - 2K_{\Sigma_t}+ |\overset{\circ}{h}|^2) + \frac{2 a + 1}{4}H^2 \right\} \\
				&-a \beta(t) f'(t) \int_{\Sigma_t} H |\nabla u|-\alpha(t) f'(t)\int_{\partial\Sigma_t}h^{\hat\partial M'}_{\nu\nu}   \\
				=& 2 \pi \gamma'(t) - \alpha(t) f'(t) \left\{  \int_{\Sigma_t} |\nabla u|^{-2}|\nabla_{\Sigma_t} |\nabla u||^2 + \frac{1}{2} ( R_{M'} - 2K_{\Sigma_t} + |\overset{\circ}{h}|^2 ) \right\} \\
				&- (2 a + 1) \alpha(t) f'(t) \int_{\Sigma_t} \left( \frac{H}{2} - \eta(u) |\nabla u| \right)^2
-\alpha(t) f'(t)\int_{\partial\Sigma_t}h^{\hat\partial M'}_{\nu\nu}\\
				&+ \left[\alpha'(t) - (2 a + 1) \eta(t) f'(t) \alpha(t) - a f'(t) \beta(t)\right]\int_{\Sigma_t} H |\nabla u|    \\
				&+ \left[\beta'(t)  + (2 a + 1) (\eta(t))^2 f'(t) \alpha(t)\right]\int_{\Sigma_t} |\nabla u|^2  .
			\end{align*}
			Using the 	system of ODEs \eqref{equ: differential equations}, we get
			\begin{align*}
				F'(t)=&\gamma'(t) \left\{  \int_{\Sigma_t} |\nabla u|^{-2}|\nabla_{\Sigma_t} |\nabla u||^2 + \frac{1}{2}R_{M'}  +\frac12 |\overset{\circ}{h}|^2 + (2 a + 1) \left( \frac{H}{2} - \eta(u) |\nabla u| \right)^2 \right\}\\
				&+\gamma'(t)\left(2\pi-\int_{\Sigma_t}  K_{\Sigma_t}-\int_{\partial\Sigma_t}h^{\hat\partial M'}_{\tau\tau}+\int_{\partial\Sigma_t}H^{\hat\partial M'}\right),
			\end{align*}
where $h^{\hat\partial M'}_{\tau\tau}$ is geodesic curvature of $\partial\Sigma_t$ for $\tau\in T(\hat\partial M')\cap T\Sigma_t$ and $H^{\hat\partial M'}$ is mean curvature
of $\hat\partial M'$ (see Lemma 5.7 in \cite{Koerber--2023}).

%%%%%%%%%%%%%%%%%%
			 { Since $M'$ is simply connected, we can see that $\Sigma_t$ is connected by Lemma \ref{lem:2.3} which implies that  $\chi(\Sigma_t)\leq 1$. It follows from the Gauss-Bonnet formula that $$\int_{\Sigma_t} K_{\Sigma_t}+ \int_{\partial\Sigma_t}h^{\hat\partial M'}_{\tau\tau}\leq 2\pi.$$
        }
In view of Proposition A.2 in \cite{Xia--Yin--Zhou--2024} and the assumption \eqref{alpha>0-assumpt}, we see that $\alpha(t)\ge 0$. The assertion follows since $\gamma'(t)=-f'(t)\alpha(t)\le 0$, $R_M\ge 0$, $H^{\hat\partial M'}\geq0$ and $2a+1=\frac{5-p}{p-1}>0$. 	\end{proof}

%%%%%%%%%%%%%%%%%%%%%%%%%%%%%%%%%%%%%%%%
\subsection{Monotonicity of $F(t)$ via regularization} \label{Sec:3.2}
		In order to prove the monotonicity part in Theorem \ref{thm:Monot}, we need to establish the monotone property of $F(t)$ via regularization by using similarly method with \cite{Xia--Yin--Zhou--2024}. Let $u$ be the solution to \eqref{p-H}.
		
		In the following we denote $$f_0(t)=f(t), \eta_0(t)=\eta(t), a_0(t)=a(t), b_0(t)=\b(t), g_0(t)=g(t).$$
		Following \cite{Agostiniani--Mantegazza--Mazzieri--Oronzio--2023, Hirsch--Miao--Tam--2024}, we approximate $u$ by smooth function $\{v_{\ve}\}_{\ve>0}$, which is a sequence of solutions of
		\begin{equation}\label{eqn:3.1}
			\left\{
			\begin{aligned}
				{\rm div}(|\n v_\ve|_{\ve}^{p-2}\n v_\ve)&=0 \ \ \  \  {\rm in}\ \ M(T)\cap M',\\
				v_\ve&=0\ \ \ \ {\rm on} \ \ \Sigma,\\
\langle \nabla v_{\ve},\mu\rangle&=0\ \ {\rm on} \ \ \hat\p M',\\
				v_\ve&=f_0(T)\ \ \  \ {\rm on}\  \ \Sigma(T),
			\end{aligned}\right.
		\end{equation}
		where $M(T)=\{0<u<f_0(T)\}$, $\Sigma(t)=\{u=f_0(t)\}$ and $|\n v_\ve|_{\ve}=\sqrt{|\n v_\ve|^2+\ve^2}$.  It is clear that for any $\ve>0$, $v_{\ve}$ is smooth. We can obtain from  \cite{DiBenedetto--1983--NATMA, DiBenedetio--1983--PAPP} that as $\ve\rightarrow 0$, $v_{\ve}\rightarrow u$ in $C^{1,\beta}$-topology  for some $\beta >0$ on any compact subsets of $M(T)\cap M'$, and $v_{\ve}\rightarrow u$ in $C^{\infty}$-topology on any compact subsets of $M(T)\setminus\{|\n u|\neq 0\}$.
		
		In the following, we'll simplify notation for convenience by omitting the subscript $\ve$ and using $v$ instead of $v_\ve$ when there is no risk of confusion. We define
		\begin{equation}\label{eqn:3.2}
			\left\{
			\begin{aligned}
				&{\rm Cap}_{p,\ve}=\int_{\p M'}|\n v|_{\ve}^{p-2}|\n v|=\int_{\Sigma_{t, \ve}}|\n v|_{\ve}^{p-2}|\n v|,\\
				&\mathfrak{c}_{p,\ve}=\left(\frac{{\rm Cap}_{p,\ve}}{2\pi}\right)^{\tfrac1{p-1}},\\
				&{m_{\ve}= \operatorname{sgn}(k) \left( \mathfrak{c}_{p,\ve}I_{a}(k) \right) ^{\frac{1}{a}},} \\
				&{r_{0,\ve}=\frac{m_\ve}{k},}\\
				&f_{\ve}(t)=1 - \int_{t}^{\infty} \mathfrak{c}_{p,\ve} s^{-a-1}\left(1+\frac{m_{\ve}}{s}\right)^{-2a},\\
				&\eta_{\ve}(t) = \mathfrak{c}_{p,\ve}^{-1}t^{a}\left(1 + \frac{m_{\ve}}{ t}\right)^{2a-1}\left(1- \frac{m_{\ve}}{t}\right).
			\end{aligned}\right.
		\end{equation}
		{One can check \begin{eqnarray}\label{xeq-2}
				a\eta_\ve(r_{0,\ve})-1<0.
			\end{eqnarray}
			In fact, when $k=1$, we have $a\eta_\ve(r_{0,\ve})-1=-1<0$. For $k\in(-1,0)\cup(0,1)$, \eqref{xeq-2} is equivalent to $aI_{a}(k)<|k|^{a}(1+k)^{-2a+1}(1-k)^{-1}$, which can be easily verified by taking the derivative with respect to $k$ on both sides of the inequality.  }
		
		Denote $\Sigma_{t, \ve}:=\{v_{\ve}=f_{\ve}(t)\}$. When $\Sigma_{t, \ve}$ is a regular hypersurface, we define $F_{\ve}(t)$ as follows:
		\begin{equation}\label{eqn:3.3}
			F_{\ve}(t) = 2 \pi \gamma_{\ve}(t) + \alpha_{\ve}(t) \int_{\Sigma_{t, \ve}} H |\nabla v| + \beta_{\ve}(t) \int_{\Sigma_{t, \ve}} |\nabla v|^2.
		\end{equation}
		Here $\alpha_{\ve}(t)$, $\beta_{\ve}(t)$, $\gamma_{\ve}(t)$
		are solutions to the corresponding systems of ODEs:
		\begin{equation}\label{eqn:3.4}
			\left\{
			\begin{aligned}
				&0=\alpha_{\ve}'(t) - (2 a + 1) \eta_{\ve}(t) f'_{\ve}(t) \alpha_{\ve}(t) - a f'_{\ve}(t) \beta_{\ve}(t), \\
				&0=\beta_{\ve}'(t) + (2 a + 1) (\eta_{\ve}(t))^2 f'_{\ve}(t) \alpha_{\ve}(t), \\
				&0=\gamma'_{\ve}(t) + f'_{\ve}(t) \alpha_{\ve}(t).
			\end{aligned}\right.
		\end{equation}
		By a similar consideration as in Propositions A.1 and A.2 in \cite{Xia--Yin--Zhou--2024}, one sees that $\alpha_\ve(t)$ is given by
		\begin{equation}\label{alpha-ep}
			\begin{aligned}
				\alpha_\ve(t)
				=& t\left(1+\frac{m_\ve}{t}\right)^{2}
				\left\{ \left( C_{2,\ve}\mathfrak{c}_{p,\ve}+C_{1,\ve}\frac{a}{m_\ve} \frac{I_{a}(\frac{m_\ve}{t})}{I_{a}(\frac{m_\ve}{r_{0,\ve}})}  \right)
				\eta_\ve(t)-C_{1,\ve}\frac{1}{m_\ve} \right\},
			\end{aligned}
		\end{equation}
		where $C_{1,\ve}$ and $C_{2,\ve}$ are two constants,
		and $\alpha_\ve(t)\ge 0$ if and only if (see \cite{Xia--Yin--Zhou--2024} for details)
		\begin{eqnarray}\label{alpha>0-assumpt1}
			C_{2,\ve} \ge 0, \hbox{ and }
			\left( C_{2,\ve}\mathfrak{c}_{p,\ve}+C_{1,\ve}\frac{a}{m_\ve} \right)
			\eta_\ve(r_{0,\ve}) \ge C_{1,\ve}\frac{1}{2m_\ve}. 	
		\end{eqnarray}	
		%By the assumption \eqref{alpha>0-assumpt} and the simple fact that $\mathfrak{c}_{p,\ve}, m_\ve, \eta_\ve(r_{0})$ converge to $\mathfrak{c}_{p}, m, \eta(r_{0})$, respectively, as $\ve\to 0$, we may choose appropriate $C_{i,\ve} \ (i=1,2)$ so that \eqref{alpha>0-assumpt1} holds and $C_{i,\ve}\to C_i$ as $\ve\to 0$ $(i=1,2)$.
		%For such choice, we can see  that $\alpha_\ve(t)\ge 0$. {See Lemma \ref{Lem:A.1} for the details.}
		%{We note that $
			%\lim\limits_{\epsilon\rightarrow0}\operatorname{Cap}_{p,\epsilon} = \operatorname{Cap}_{p}
		%	$.
		%	Since $\mathfrak{c}_{p,\ve}, m_\ve, \eta_\ve(r_{0,\ve})$ are continuous about $\operatorname{Cap}_{p,\epsilon}$, one sees that $\mathfrak{c}_{p,\ve}, m_\ve, \eta_\ve(r_{0,\ve})$ converge to $\mathfrak{c}_{p}, m, \eta(r_{0})$, respectively, as  $\ve\to 0$. Then we choose  $C_{i,\ve}, i=1,2$ as follows:
		%	\begin{eqnarray}\label{equ:Ciepsilon}
			%	\begin{cases}
			%		C_{2,\epsilon}\mathfrak{c}_{p,\epsilon} = C_{2}\mathfrak{c}_{p}, \\
			%		\left( C_{2,\epsilon}\mathfrak{c}_{p,\epsilon}+C_{1,\epsilon}\frac{a}{m_{\epsilon}} \right)
			%		\eta_{\epsilon}(r_{0,\ve})-C_{1,\epsilon}\frac{1}{m_{\epsilon}} = \left( C_{2}\mathfrak{c}_{p}+C_{1}\frac{a}{m} \right)
					%\eta(r_{0})-C_{1}\frac{1}{m}.
			%	\end{cases}
		%	\end{eqnarray}
			%This can be done because of \eqref{xeq-2}. By such choice of $C_{i,\ve}, i=1,2$ and the assumption \eqref{alpha>0-assumpt}, we get \eqref{equ:Ciepsilon}, and hence $\alpha_\ve(t)\ge 0$. Moreover,
		%	$\lim\limits_{\epsilon\rightarrow0}C_{i,\epsilon} = C_{i}.$}
		
%%%%%%%%%%%%%%%%%%%%%%%%%
		
		Next, we use $\alpha_{\ve}(v),\beta_{\ve}(v), \gamma_{\ve}(v), \eta_{\ve}(v)$ to indicate $\alpha_{\ve}(t), \beta_{\ve}(t), \gamma_{\ve}(t), \eta_{\ve}(t)$ for $t=f_{\ve}^{-1}(v)$, respectively. Thus $\alpha_{\ve}(v),\beta_{\ve}(v), \gamma_{\ve}(v)$, as functions of $v$, satisfy that
		\begin{equation}\label{eqn:3.5}
			\left\{
			\begin{aligned}
				&0=\alpha_{\ve}'(v) - (2 a + 1) \eta_{\ve}(v) \alpha_{\ve}(v) - a \beta_{\ve}(v), \\
				&0=\beta_{\ve}'(v) + (2 a + 1) (\eta_{\ve}(v))^2  \alpha_{\ve}(v), \\
				&0=\gamma'_{\ve}(v) +  \alpha_{\ve}(v).
			\end{aligned}\right.
		\end{equation}
		It is easy to see that
		\begin{eqnarray}\label{equ-Delta u}
			\Delta v=(2-p)\frac{|\n v|^2}{|\n v|^2_{\ve}}v_{\nu\nu},
		\end{eqnarray}
		where $v_{\nu\nu}=\frac{g(\nabla|\nabla v|,\nabla v)}{|\nabla v|}$, and the mean curvature $H$ of $\Sigma_{t, \ve}$ is given by
		\begin{eqnarray}\label{equ-H}
			H=\frac1{|\n v|}(\Delta v-v_{\nu\nu})=-\frac1{|\n v|}\frac{(p-1)|\n v|^2+\ve^2}{|\n v|^2_{\ve}}v_{\nu\nu}.
		\end{eqnarray}
		%see for example \cite[Lemma 3.1]{HMT}.
		Using \eqref{eqn:3.2} and \eqref{equ-H}, we can write
		\begin{equation*}
			\begin{aligned}
				F_{\ve}(t) %=& 4 \pi \gamma_{\ve}(t) + \alpha_{\ve}(t) \int_{\Sigma_{t, \ve}} H |\nabla v| d \sigma + \beta_{\ve}(t) \int_{\Sigma_{t, \ve}} |\nabla v|^2 d\sigma\\
				=&\int_{\Sigma_{t, \ve}}2\pi\gamma_{\ve}(v) {\rm Cap}_{p,\ve}^{-1}|\n v|^{p-2}_{\ve}|\n v|+\alpha_{\ve}(v)(\Delta v-v_{\nu\nu})+\beta_{\ve}(v)|\n v|^2.
			\end{aligned}
		\end{equation*}
%%%%%%%%%%%%%%%%%%%%%%%%%%		

		Let
		$
		X_{\ve}=U_{\ve}+V_{\ve}+W_{\ve}
		$
		where
		\begin{equation*}
			\left\{
			\begin{aligned}
				U_{\ve}=&2\pi\gamma_{\ve}(v) {\rm Cap}_{p,\ve}^{-1}|\n v|^{p-2}_{\ve}\n v,\\
				V_{\ve}=&\alpha_{\ve}(v)\left(\frac{\Delta v}{|\n v|}\n v-\n |\n v|\right),\\
				W_{\ve}=&\beta_{\ve}(v)|\n v|\n v.
			\end{aligned}\right.
		\end{equation*}	
		%	and $X=\lim_{\ve\rightarrow0}X_{\ve}$,
		Then
		\begin{equation}\label{eqn:3.6}
			F_{\ve}(t)=\int_{\Sigma_{t, \ve}}\left\langle X_{\ve},\, \frac{\n v}{|\n v|}\right\rangle.
		\end{equation}
		By adapting the proof of \cite[Lemma 1.3]{Agostiniani--Mantegazza--Mazzieri--Oronzio--2023}, since $\alpha_{\ve}(t), \beta_{\ve}(t), \gamma_{\ve}(t)$ converge to $\alpha(t), \beta(t), \gamma(t)$, respectively, as $\ve\to 0$, we have the following lemma.
		\begin{lemma}\label{lem:3.2}
			Suppose $\{u=f_{0}(t)\}$ is regular for $f_0(t)\in(0, f_0(T))$. Then for  $\ve>0$ small enough, $\Sigma_{t, \ve}=\{v_{\ve}=f_{\ve}(t)\}$ is also regular. Moreover,
			\begin{equation*}
				\lim_{\ve\rightarrow 0}F_{\ve}(t)=F(t).
			\end{equation*}
		\end{lemma}

		For $\delta>0$, let
		\begin{equation*}
			\left\{
			\begin{aligned}
				%U_{\ve,\delta}=&U_{\ve}=4\pi\gamma_{\ve}(v) {\rm Cap}_{p,\ve}^{-1}|\n v|^{p-2}_{\ve}\n v,\\
				V_{\ve,\delta}=&\alpha_{\ve}(v)\left(\frac{\Delta v}{|\n v|_{\delta}}\n v-\n |\n v|_{\delta}\right),\\
				W_{\ve,\delta}=&\beta_{\ve}(v)|\n v|_{\delta}\n v,\\
				X_{\ve,\delta}=&U_{\ve}+V_{\ve,\delta}+W_{\ve,\delta}.
			\end{aligned}\right.
		\end{equation*}
		It is clear that $U_{\ve}, V_{\ve,\delta}, W_{\ve,\delta}$ are smooth in $M_T$. Let $t_1 < t_2 $ such that $\Sigma_{t_1, \ve},\Sigma_{t_2, \ve}$ are regular.
		One sees from the divergence theorem and \eqref{eqn:3.6} that
		\begin{equation}\label{eqn:3.7}
			\begin{aligned}
				F_{\ve}(t_2)-F_{\ve}(t_1)=&\lim_{\delta\rightarrow0}\left(\int_{\{f_\ve(t_1)<v<f_\ve(t_2)\}}{\rm div} X_{\ve,\delta}-\int_{\hat\partial M'_{\ve}\cap\Omega_{\ve}}\langle X_{\ve,\delta},\mu \rangle\right),
			\end{aligned}
		\end{equation}
where $\Omega_{\ve}=\{f_\ve(t_1)<v<f_\ve(t_2)\}$.

		Next we compute the intergrand in the right hand side of \eqref{eqn:3.7}.
		\begin{lemma}[(Lemma 3.5 of \cite{Xia--Yin--Zhou--2024})]\label{lem:3.3}\
			\begin{itemize}
\item [(i)]  At the points where $|\n v|=0$, we have $\langle X_{\ve,\delta},\mu\rangle=0$. At the points where $|\n v|>0$, we have $\langle X_{\ve,\delta},\mu\rangle=\alpha_{\ve}(v)\frac{|\nabla u|^2}{|\nabla u|_{\delta}}h^{\hat\partial M'}_{\nu\nu}$.
				\item[(ii)] ${\rm div} U_{\ve}=2\pi \gamma'_{\ve}(v) {\rm Cap}_{p,\ve}^{-1}|\n v|_{\ve}^{p-2}|\n v|^2$.
				
				\item[(iii)] At the points where $|\n v|=0$, we have ${\rm div}W_{\ve,\delta}=\beta_{\ve}(v)\delta\Delta v$.
				
				At the points where $|\n v|>0$, we have
				$$
				{\rm div}W_{\ve,\delta}=\beta_{\ve}(v)\left((2-p)\frac{|\n v|_{\delta}|\n v|^2}{|\n v|^2_{\ve}}v_{\nu\nu}+\frac{|\n v|^2}{|\n v|_{\delta}}v_{\nu\nu}\right)+\beta'_{\ve}(v)|\n v|^2|\n v|_{\delta}.
				$$
				
				\item[(iv)] At the points where $|\n v|=0$,  we have ${\rm div}V_{\ve,\delta}\leq0$.
				
				At the points where $|\n v|>0$,
				we have ${\rm div}V_{\ve,\delta}\leq I_{\ve,\delta}$,
				where
				\begin{equation*}
					\begin{aligned}
						I_{\ve,\delta}=&\alpha_{\ve}(v)|\n v|^{-1}_{\delta}\Bigg\{(2-p)^2\frac{|\n v|^4}{|\n v|_{\ve}^4}v_{\nu\nu}^2-(2-p)\frac{|\n v|^4}{|\n v|_{\ve}^2|\n v|_{\delta}^2}v_{\nu\nu}^2-|\n v|^2{\rm Ric}(\nu,\nu)\\
						&-\frac{1}{2}\frac{|\nabla v|^2}{|\nabla v|_{\ve}^2} \left( (p-2)^2\frac{|\nabla v|^2}{|\nabla v|_{\ve}^2} + 2p-3 \right)v_{\nu\nu}^2\Bigg\}\\
						&+\alpha'_{\ve}(v)\left((2-p)\frac{|\n v|^4}{|\n v|^2_{\ve}|\n v|_{\delta}}-\frac{|\n v|^2}{|\n v|_{\delta}}\right)v_{\nu\nu}.
					\end{aligned}
				\end{equation*}
			\end{itemize}
		\end{lemma}
%%%%%%%%%%%
\begin{proof}
Since $\langle \nabla u,\mu\rangle=0$, it is easy to see that
$$\langle X_{\ve,\delta},\mu\rangle=-\alpha_{\ve}(v)\langle \nabla|\nabla v|_{\delta},\mu\rangle=-\alpha_{\ve}(v)\frac{|\nabla u|}{|\nabla u|_{\delta}}\langle \nabla|\nabla v|,\mu\rangle.$$
If $|\nabla u|=0$,  it is easy to see that $\langle X_{\ve,\delta},\mu\rangle=0$. If $|\nabla u|>0$, from $\frac{d}{dt}\langle \nabla u,\mu\rangle=0$, we have $\langle \nabla|\nabla v|,\mu\rangle=-|\nabla v|h^{\hat\partial}_{\nu\nu}$. It follows that $$\langle X_{\ve,\delta},\mu\rangle=\alpha_{\ve}(v)\frac{|\nabla u|^2}{|\nabla u|_{\delta}}h^{\hat\partial M'}_{\nu\nu}.$$
From Lemma 3.5 of \cite{Xia--Yin--Zhou--2024}, we can see that (ii)-(iv) hold. This complete the proof.
\end{proof}
%%%%%%%%%%%%%%%%%%%%%%%

		\begin{proposition}\label{lem:3.4}
			Let
			$\{u=f_0(t_1)\}$, $\{u=f_0(t_2)\}$ be two regular level sets for $t_1< t_2$.
			Assume $\alpha_{\ve}(t)\ge 0$ on {$(\frac{|m_{\ve}|}{2},+\infty)$}. Then the following inequality holds:
			\begin{eqnarray}\label{asymp-monot}
				F_{\ve}(t_2)-F_{\ve}(t_1)%\leq&\int_{\tau_1}^{\tau_2}\gamma'_{\ve}(\tau)\left(4\pi -\int_{\Sigma(\tau)} K\right)
				\leq \ve\int_{\{f_\ve(t_1)<v<f_\ve(t_2)\}}\frac{|\alpha'_{\ve}(v)-\beta_{\ve}(v)|^2}{\alpha_{\ve}(v)}|\n v|.
			\end{eqnarray}
		\end{proposition}
		
		\begin{proof}
			
			By Lemma \ref{lem:3.3}, for any $\delta>0$, we have
			\begin{equation}\label{4.21}
				\begin{aligned}
					F_{\ve}(t_2)-F_{\ve}(t_1)=&\int_{\{f_\ve(t_1)<v<f_\ve(t_2)\}}{\rm div}X_{\ve,\delta}-\int_{\hat\partial M'_{\ve}}\langle X_{\ve,\delta},\mu \rangle\\ \leq&\int_{\{f_\ve(t_1)<v<f_\ve(t_2)\}}{\rm div}U_{\ve}+C\delta+\int_{\{f_\ve(t_1)<v<f_\ve(t_2)\}}({\rm div}W_{\ve,\delta}+I_{\ve,\delta})\mathbf{1}_{\{|\n v|>0\}}\\
&-\int_{\hat\partial M'_{\ve}\cap\Omega_{\ve}}\langle X_{\ve,\delta},\mu \rangle,
				\end{aligned}
			\end{equation}
			for some $C > 0$ independent of $\delta$, where $\mathbf{1}_K$ is the characteristic function of $K$.
From  3.26-3.31 of Proposition 3.2 in \cite{Xia--Yin--Zhou--2024}, we have
\begin{equation}\label{4.22}
\begin{aligned}
&\int_{\{f_\ve(t_1)<v<f_\ve(t_2)\}}{\rm div}U_{\ve}+C\delta+\int_{\{f_\ve(t_1)<v<f_\ve(t_2)\}}({\rm div}W_{\ve,\delta}+I_{\ve,\delta})\mathbf{1}_{\{|\n v|>0\}}	\\
&\leq\ve\int_{\{f_\ve(t_1)<v<f_\ve(t_2)\}}\frac{|\alpha'_{\ve}(v)-\beta_{\ve}(v)|^2}{\alpha_{\ve}(v)}|\n v|- \int_{f_\ve(t_1)}^{f_\ve(t_2)}\alpha_{\ve}(\xi)
					\left( 2\pi -\int_{\Sigma_{t,\ve}} K \right) d\xi
\end{aligned}
\end{equation}	
Substituting \eqref{4.22} into \eqref{4.21}, by using co-area formula and property (i)  in Lemma \ref{lem:3.3}, yields to	
			\begin{equation*}\label{eqn:3.26}
				\begin{aligned}
					F_{\ve}(t_2)-F_{\ve}(t_1)\leq &\ve\int_{\{f_\ve(t_1)<v<f_\ve(t_2)\}}\frac{|\alpha'_{\ve}(v)-\beta_{\ve}(v)|^2}{\alpha_{\ve}(v)}|\n v|\\
&- \int_{f_\ve(t_1)}^{f_\ve(t_2)}\alpha_{\ve}(\xi)
					\left( 2\pi -\int_{\Sigma_{t,\ve}} K-\int_{\partial\Sigma_{t,\ve}}h^{\Gamma}_{\tau\tau}+\int_{\partial\Sigma_{t,\ve}}H^{\hat\partial M'} \right) d\xi\\
				\leq&\ve\int_{\{f_\ve(t_1)<v<f_\ve(t_2)\}}\frac{|\alpha'_{\ve}(v)-\beta_{\ve}(v)|^2}{\alpha_{\ve}(v)}|\n v|,
				\end{aligned}
			\end{equation*}
			where we have used $2\pi -\int_{\Sigma_{t,\ve}} K-\int_{\partial\Sigma_{t,\ve}}h^{\Gamma}_{\tau\tau}\ge 0$, for each regular level set $\{v=\xi\}$ is connected.
			This completes  the proof of Proposition \ref{lem:3.4}.
		\end{proof}
		
		By letting $\ve\to 0$ in \eqref{asymp-monot}, in view of Lemma \ref{lem:3.2}, we see the following
		\begin{corollary} Let
			$\{u=f_0(t_1)\}$, $\{u=f_0(t_2)\}$ be two regular level sets for $t_1< t_2$. Then $F(t_2)\le F(t_1)$.
		\end{corollary}
		This finishes the proof of monotonicity part in  Theorem \ref{thm:Monot}.

\section{Asymptotic behavior of $F(t)$}\label{Sec:05}

In this section, we will estimate the limit of $F(t)$ as $t\rightarrow+\infty$.

Firstly,	by using \eqref{ode-solution1}  we may rewrite $F(t)$ as follows:
		\begin{eqnarray}\label{equ:F1}
			F(t)
			&=
			& - \left( C_{2}\mathfrak{c}_{p}\frac{2m}{a}+C_{1}\frac{I_{a}(\frac{m}{t})}{I_{a}(\frac{m}{r_{0}})} \right)
			\left(
			2\pi - \mathfrak{c}_{p}^{-2} t^{2a}(1+\frac{m}{t})^{4a} \int_{\Sigma_{t}} |\nabla u|^2
			\right) \\
			&& - \frac{\alpha(t)}{4\eta(t)} \int_{\Sigma_{t}} \left( H-2\eta(t)|\nabla u| \right) ^2
			- \frac{1}{4} \frac{\alpha(t)}{\eta(t)}
			\left( 8\pi - \int_{\Sigma_{t}} H^2 \right) \nonumber
			\\&&-2\pi \frac{\alpha(t)}{\eta(t)}\left( \frac{(1-\frac{m}{t})^{2}}{(1+\frac{m}{t})^{2}}-1 \right). \nonumber
		\end{eqnarray}
		In order to estimate the limit of $F(t)$, we need to know the limit of some relevant quantities. We follows
Xia-Yin-Zhou's idea to obtain the following results.
		%%%%%%
		\begin{lemma}
			\begin{eqnarray}
				&&\lim_{t\to\infty} \left( \frac{\alpha(t)}{\eta(t)}-C_{2}\mathfrak{c}_{p}t(1 + \frac{m}{t})^{2} \right) =0, \label{equ:EstimateAlphaDevidedByEta}\\
				&&\lim_{t\to\infty} 2\pi \frac{\alpha(t)}{\eta(t)}\left( \frac{(1-\frac{m}{t})^{2}}{(1+\frac{m}{t})^{2}}-1 \right)
				= -8 \pi C_{2}\mathfrak{c}_{p}m.\label{xeq-limit6}
			\end{eqnarray}
		\end{lemma}
		%%%%%%
		%%%%%%
		\begin{proof}
	We just have to instead of  $m$ by $2m$ in the proof of Lemma 5.1 in \cite{Xia--Yin--Zhou--2024}.
			
		\end{proof}
		%%%%%%%
		
		%%%%%%
		Similar to \cite[Lemma 5.2]{Xia--Yin--Zhou--2024}, we have the following facts.
		\begin{lemma}
			%Under the assumptions of Theorem \ref{Thm:1.01},
			Assume that $f(t)$ is a regular value of $u$. Let $0 < \tilde{\tau} < \min \{ \tau,1 \}$. Then, along $\Sigma_{t} = \{ u = f(t) \}$, we have  that as $t \rightarrow + \infty$,
			\begin{eqnarray}
				&&|\Sigma_{t}| = 2 \pi   t^{2} \left( 1 + O(t^{- \tilde{\tau}}) \right),\label{equ:Area}
				\\&&\int_{\Sigma_t} H |\nabla u| = 4 \pi \mathfrak{c}_{p}  t^{- a} \left( 1 + O(t^{-  \tilde{\tau}}) \right),\label{equ:IntegrateHNablau}	
				\\&&\int_{\Sigma_t} |\nabla u|^{2} = 2 \pi \mathfrak{c}_{p}^{2} t^{- 2 a} \left( 1 + O(t^{-  \tilde{\tau}}) \right). \label{equ:IntegrateNablauSquare}
			\end{eqnarray}
		\end{lemma}
		%%%%%%
		%%%%%%
		
	This combining with Lemma 5.3 and 5.4 of \cite{Xia--Yin--Zhou--2024} yields that
\begin{lemma}
			\begin{equation} \label{equ:LimitofF1}
				\lim\limits_{t\rightarrow+\infty}
				\left( 2\pi-\mathfrak{c}_{p}^{-2} t^{2a}(1+\frac{m}{t})^{4a} \int_{\Sigma_{t}} |\nabla u|^2   \right)= 0.
			\end{equation}
			\begin{equation} \label{equ:LimitofF2}
				\lim_{t \rightarrow \infty} \frac{\alpha(t)}{\eta(t)} \int_{\Sigma_{t}} \left( \frac{H}{2} - \eta(u) |\nabla u| \right)^{2} =0
			\end{equation}
		\end{lemma}
		%%%%%%
		
		%%%%%%
		%%%%%%
		
		%%%%%%
		\begin{lemma}Assume $C_2\ge 0$. Then
			%Suppose that $u$ is a solution of \eqref{p-H},
			\begin{equation} \label{equ:LimitofF3}
				\lim_{t \rightarrow \infty} \frac14\frac{\alpha(t)}{\eta(t)}\left( 8 \pi - \int_{ \Sigma_t} H^{2}   \right) \le 8 \pi C_{2}\mathfrak{c}_{p} \mathfrak{m}_{ABL}.
			\end{equation}
			Furthermore, when $C_{2}=0$,
			$$
			\lim\limits_{t \rightarrow \infty} \frac14\frac{\alpha(t)}{\eta(t)}\left( 8 \pi - \int_{  \Sigma_t} H^{2} \right) =0.
			$$
		\end{lemma}
		%%%%%%
		
		%%%%%%
		\begin{proof}
	{\bf Step 1.}		In the same spirit of \cite[Lemma 2.5]{Agostiniani--Mantegazza--Mazzieri--Oronzio--2023}, we proved that
			\begin{eqnarray}\label{limit-ammo}
				\lim_{t \rightarrow \infty}M_p(t)=\lim_{t \rightarrow \infty} \frac{t}{4}\left( 8 \pi - \int_{ \left\{ u=f(t) \right\} } H^{2} \right) \le 8 \pi \mathfrak{m}_{ABL}.
			\end{eqnarray}

We are going to compare the expression of $M_p$ with an analogous expression in which the geometric quantities are computed with respect to the Euclidean background metric.
In order to do that, we work in an asymptotically flat coordinate chart. As the unit normal vectors to a regular level set $\Sigma_t$ are given by
\begin{equation*}
\nu_g=\frac{\nabla u}{|\nabla u|_g},\ \ \ \ \nu=\frac{\nabla u}{|\nabla u|_{\bar g}}
\end{equation*}
the mean curvatures can be  computed,
\begin{equation*}
H_g=(g^{ij}-\nu^i\nu^j)\frac{\nabla_i\nabla_j u}{|\nabla u|_g}\ \ \ \ H_{\bar g}=(\delta^{ij}-\nu^i\nu^j)\frac{\nabla_i\nabla_j u}{|\nabla u|}.
\end{equation*}

In \cite[Lemma 2.5]{Agostiniani--Mantegazza--Mazzieri--Oronzio--2023}, they prove that
\begin{equation*}
H^2_gd\sigma_g=\left(H^2_{\bar g}-\frac2t{\rm div}_{\Sigma_t}\omega^T-\frac2t\delta^{ij}(\partial_ig_{jk}-\partial_k g_{ij})\nu^k+O(t^{-2-2\beta})\right)d\sigma,\ \ \ \ {\rm on}\ \ \Sigma_t,
\end{equation*}
where $\omega=\gamma_{jk}\nu^jdx^k=\omega^{T}+\gamma(\nu,\nu)\frac{du}{|\nabla u|}$.

Then, we have
\begin{equation*}
\begin{aligned}
M_p(t)=&\frac{t}4\left(8\pi-\int_{\Sigma_t}H^2_{\bar g}d\sigma\right)\\
=&\frac{t}4\left(8\pi-\int_{\Sigma_t}H^2_{\bar g}d\sigma\right)+\frac{1}2\int_{\partial\Sigma_t}g_{i\nu}\mu^i+\frac12\int_{\Sigma_t}\delta^{ij}(\partial_ig_{jk}-\partial_k g_{ij})\nu^k+O(t^{1-2\beta})
\end{aligned}
\end{equation*}
Thus, for $\beta>1/2$, by the definition of ABL mass and Willmore inequality in \cite{Jia--Wang--Xia--Zhang--2024},  we have $$\lim_{t\rightarrow\infty}M_p(t)\leq8\pi\mathfrak{m}_{ABL}.$$

	{\bf Step 2.}		The assertion follows from \eqref{limit-ammo} and \eqref{equ:EstimateAlphaDevidedByEta}.
			%	i.e. $s$ and $t$ are equivalent infinitesimal quantities.
			%	This lemma can be obtained directly by the above lemma.
			When $C_{2}=0$, recalling \eqref{equ:Area} and $H=\frac{2}{r} ( 1 + O(r^{- \tilde{\tau}})$, we get
			$$
			\lim\limits_{t \rightarrow \infty} \left( 8 \pi - \int_{ \Sigma_t} H^{2}\right) =0.
			$$
			The second assertion is proved by using \eqref{equ:EstimateAlphaDevidedByEta}.
		\end{proof}	

Combining  \eqref{xeq-limit6}, \eqref{equ:LimitofF1},  \eqref{equ:LimitofF2} and  \eqref{equ:LimitofF3}, we get the limit of $F(t)$.
		%%%%%%
		\begin{proposition}\label{limit-F}
			Assume $C_2\ge 0$. Then
			\begin{equation} \label{equ:LimitofF}
				\begin{aligned}
					\lim\limits_{t\rightarrow+\infty} F(t)
					\ge -8 \pi C_{2}\mathfrak{c}_{p} (\mathfrak{m}_{ABL}-m).
				\end{aligned}
			\end{equation}
			Furthermore, when $C_{2}=0$,
			$$
			\lim\limits_{t\rightarrow+\infty} F(t)=0.
			$$
		\end{proposition}
%%%%%%%%%%%%%%%%%%%
\section{Applications and proof of main results}\label{Sec:06}
The proof of the rigidity part in Theorem \ref{thm:Monot} is similarly with Section 4 in \cite{Xia--Yin--Zhou--2024}.  We omit the proof in here.

		Combining Theorem \ref{thm:Monot} and Proposition \ref{limit-F}, we get the following corollary.
		%%%%%%
		\begin{corollary} \label{Cor:6.1}
Let $(M',g)$ be a $3$-dimensional, complete, one-end asymptotically flat half space with boundary $\partial M'=\hat\partial M'\bigcup\Sigma$ that has nonnegative scalar curvature and $H^{\hat\partial M'}\geq0$, where $\Sigma$ is a free boundary. Assume that $M'$ is simply connected. Let $p\in (1, 3)$ and $u$ be the weak solution to \eqref{p-H}. {For any $k\in(-1,0)\cup(0,1]$,
				let $m= \operatorname{sgn}(k) \left(I_a(k)\mathfrak{c}_{p}\right)^{\frac{1}{a}}$, $r_{0}=\frac{m}{k}=|k|^{-1}\left(I_a(k)\mathfrak{c}_{p}\right)^{\frac{1}{a}}$.} Let $\a, \b, \g$ be three one-variable functions given by \eqref{ode-solution1} with  $$
			C_{2} \ge 0, \hbox{ and }
			\left( C_{2}\mathfrak{c}_{p}+C_{1}\frac{a}{2m} \right)
			\eta(r_{0}) \ge C_{1}\frac{1}{2m}. 	$$	Then	we have that
			\begin{equation*}
				\begin{aligned}
					F(t)=2 \pi \gamma(t) + \alpha(t) \int_{\Sigma_{t}} H |\nabla u|   + \beta(t) \int_{\Sigma_{t}} |\nabla u|^2
					\ge  -C_{2}\mathfrak{c}_{p}8\pi(\mathfrak{m}_{ABL}-m).
				\end{aligned}
			\end{equation*}
			Moreover, equality  holds if and only if $(M', g)$ is isometric to  the  Schwarzschild half space of mass $m$ outside a rotationally symmetric half ball, $(\mathcal{M}_{m, r_0}^{3,+},  g_{m})$.	
		\end{corollary}
		%%%%%%
		
		Next, we consider the following cases, including $k=1$, $0<k<1$ and $-1<k<0$.
		
		\noindent{\bf Case 1: $k=1$.}

		In this case, $m$ and $r_{0}$ satisfy  $m= \left(I_a(1)\mathfrak{c}_{p}\right)^{\frac{1}{a}}$ and $r_0=m$. We see that
		%%%%%%
		\begin{eqnarray*}
			&&\eta(m)=0,\quad \alpha(m) =-4C_{1},\\
			&&\beta(m)
			= \left( C_{2}\mathfrak{c}_{p}\frac{2m}{a}+C_{1} \right)
			2^{4a} (I_{a}(1))^{2},\\
			&&\gamma(m)
			= - \left( C_{2}\mathfrak{c}_{p}\frac{2m}{a}+C_{1} \right) .
		\end{eqnarray*}
		%%%%%%
		Hence we have the following consequence.
		\begin{corollary}\label{prop:6.2}
			Under assumptions of Corollary \ref{Cor:6.1}.   Let $\a, \b, \g$ be three one-variable functions given by \eqref{ode-solution1} with $$C_2\ge 0,\quad C_1\le 0.$$
			Then we have that
			\begin{equation}\label{ineq-cor}
				\begin{aligned}
					&- 2 \pi \left( C_{2}\mathfrak{c}_{p}\frac{2m}{a}+C_{1} \right)
					- 2C_{1} \int_{\Sigma} H |\nabla u|
					+ \left( C_{2}\mathfrak{c}_{p}\frac{2m}{a}+C_{1} \right) 2^{4a} (I_{a}(1))^{2} \int_{\Sigma} |\nabla u|^2 \\
					\ge & -C_{2}\mathfrak{c}_{p}8\pi(\mathfrak{m}_{ABL}-m).
				\end{aligned}
			\end{equation}
			Moreover, equality holds if and only if $(M', g)$ is isometric to  the  Schwarzschild half space of mass $m$.
		\end{corollary}

\noindent{\it Proof of Theorem \ref{Thm:1.01}.}
		
		Let $C_{2}=0,C_{1}=-1$ in \eqref{ineq-cor},  we obtain \eqref{geom-ineq-1}. Let $C_{2}=\mathfrak{c}_{p}^{-1}$ and $C_{1}=0$ in \eqref{ineq-cor}, we obtain \eqref{geom-ineq-2}. \qed%Inequality \eqref{geom-ineq-3} follows from  \eqref{geom-ineq-1} and \eqref{geom-ineq-2}.
		%This completes the proof of Theorem \ref{Thm:1.01}.
		
		%\begin{Corollary}
		%		Under the assumptions of \ref{Thm:1.01},
		%		\begin{equation*}
			%			\begin{aligned}
				%%%				- 2^{4a} (I_{a}(1))^{2} \int_{\Sigma} |\nabla u|^2 d \Sigma
				%\ge 0.
				%		\end{aligned}
			%	\end{equation*}
		%	\begin{equation*}
			%		\begin{aligned}
				%%%%	\end{aligned}
			%\end{equation*}
			%\begin{equation*}
			%\begin{aligned}
			%4 \pi a
			%- \int_{\Sigma} H|\nabla u| d \Sigma
			%\le 4\pi a \frac{\mathfrak{m}_{ADM}}{m}.
			%R\end{aligned}
		%\end{equation*}
		%And $'='$ holds if and only if $(M, g)$ is isometric to a spatial Schwarzschild manifold of mass $\mathfrak{m}_{ADM}>0$.
		%\end{Corollary}
		
		\
		
		\noindent{\it Proof of Theorem \ref{Thm:1.03}.}
		Since $H=0$ on $\Sigma$, we see from \eqref{geom-ineq-1} and \eqref{geom-ineq-2} that
		\begin{equation*}
			\begin{aligned}
				0\le
				2 \pi
				- 2^{4a} (I_{a}(1))^{2} \int_{\Sigma} |\nabla u|^2
				\le 4\pi a \left( \frac{\mathfrak{m}_{ABL}}{m}-1 \right) .
			\end{aligned}
		\end{equation*}
		This is \eqref{geom-ineq-6}. On the other hand,
		By H\"older's inequality,
		\begin{equation*}
			\begin{aligned}
				2\pi \mathfrak{c}_{p}^{p-1}
				= \int_{\Sigma}|\nabla u|^{p-1}
				\le \left( \int_{\Sigma}|\nabla u|^{2} \right)^{\frac{p-1}{2}} |\Sigma|^{\frac{3-p}{2}}
				\le \left( \frac{2 \pi}{2^{4a} (I_{a}(1))^{2}} \right)^{\frac{p-1}{2}} |\Sigma|^{\frac{3-p}{2}} .
			\end{aligned}
		\end{equation*}
		It follows that
		\begin{equation*}
			\sqrt{\frac{|\Sigma|}{32\pi}} \ge  \left( I_{a}(1) \mathfrak{c}_{p} \right)^{\frac{1}{a}}.
		\end{equation*}
		This is \eqref{geom-ineq-7}.
		\qed
		%%%%%%
		
		\
		
		\noindent{\bf Case 2: $-1<k<0$ and $0<k<1$.}
		
		In this case, $m$ and $r_{0}$ satisfies that $m= \operatorname{sgn}(k) \left(I_a(k)\mathfrak{c}_{p}\right)^{\frac{1}{a}}$ and $r_0=\frac{m}{k}$. We see that
		%%%%%%
		\begin{eqnarray*}
			&&\eta(r_{0})=I_{a}(k)|k|^{-a}(1+k)^{2a-1}(1-k), \\
			&&\alpha(r_{0})
			=m\frac{(1+k)^{2}}{k}
			\left\{
			\left( C_{2}\mathfrak{c}_{p}+C_{1}\frac{a}{2m} \right) \eta(r_{0}) -C_{1}\frac{1}{2m}
			\right\}
			,\\
			&&\beta(r_{0})
			= -\eta(r_{0})\alpha(r_{0})
			+ \left( C_{2}\mathfrak{c}_{p}\frac{2m}{a}+C_{1} \right) (I_{a}(k))^{2}|k|^{-2a}(1+k)^{4a} ,\\
			&&\gamma(r_{0})
			= -(I_{a}(k))^{-2}|k|^{2a}(1+k)^{-4a}\eta(r_0)\alpha(r_{0})
			- \left( C_{2}\mathfrak{c}_{p}\frac{2m}{a}+C_{1} \right)  .
		\end{eqnarray*}
		%%%%%%
		Then we have the following consequence.
		%%%%%%
		\begin{corollary}\label{prop:6.3}
			Under assumptions of Corollary \ref{Cor:6.1}.
			{For any $k\in(-1,0)\cup(0,1)$,
				let $m=\operatorname{sgn}(k) \left(I_a(k)\mathfrak{c}_{p}\right)^{\frac{1}{a}}$, $r_{0}=\frac{m}{k}=|k|^{-1}\left(I_a(k)\mathfrak{c}_{p}\right)^{\frac{1}{a}}$.} Let $\a, \b, \g$ be three one-variable functions given by \eqref{ode-solution1} with
			$$
			C_{2} \ge 0, \hbox{ and }
			\left( C_{2}\mathfrak{c}_{p}+C_{1}\frac{a}{2m} \right)
			\eta(r_{0}) \ge C_{1}\frac{1}{2m}. 			
			$$
			Then we have that
			\begin{equation}\label{ineq-cor2}
				\begin{aligned}
					&- 2 \pi \left\{ \frac{(1-k)^{2}}{(1+k)^{2}\eta(r_{0})}\alpha(r_{0})
					+ \left( C_{2}\mathfrak{c}_{p}\frac{2m}{a}+C_{1} \right) \right\}
					+ \alpha(r_{0}) \int_{\Sigma} H |\nabla u|  \\
					&+ \left( -\eta(r_{0})\alpha(r_{0})
					+ \left( C_{2}\mathfrak{c}_{p}\frac{2m}{a}+C_{1} \right) \frac{(1+k)^{2}(\eta(r_{0}))^{2}}{(1-k)^{2}} \right)  \int_{\Sigma} |\nabla u|^2  \\
					\ge & -C_{2}\mathfrak{c}_{p}8\pi(\mathfrak{m}_{ABL}-m).
				\end{aligned}
			\end{equation}
			Moreover, equality holds if and only if $(M', g)$ is isometric to   the  Schwarzschild half space of mass $m$ outside a rotationally symmetric half ball, $(\mathcal{M}_{m, r_0}^{3,+},  g_{m})$.
		\end{corollary}
		%%%%%%
		
		\
		
		\noindent{\it Proof of Theorem \ref{Thm:1.02}.}
		Notice that for any $k \in (-1,0)\cup(0,1)$, the inequality $\frac{1-a\eta(r_{0})}{m}>0$ holds.
		Let $C_{2}=0$ and $C_{1}=-1$ in \eqref{ineq-cor2},  we obtain \eqref{geom-ineq-10}. Let $C_{2}=\mathfrak{c}_{p}^{-1}$ and $C_{1}=\frac{2m\eta(r_{0})}{1-a\eta(r_{0})}$ in \eqref{ineq-cor}, we see that $$\left( C_{2}\mathfrak{c}_{p}+C_{1}\frac{a}{2m} \right)
		\eta(r_{0})-C_{1}\frac{1}{2m}=0.$$ It follows that $\alpha(r_0)=0$. Then we easily obtain \eqref{geom-ineq-20}. \qed
		%Inequality \eqref{geom-ineq-30} follows from  \eqref{geom-ineq-10} and \eqref{geom-ineq-20}.
		%This completes the proof of Theorem \ref{Thm:1.02}.
		
		\
		
		\noindent{\it Proof of Theorem \ref{Thm:1.04}.} 	Let {$k\in(-1,1]$} be such that
		\begin{eqnarray}\label{k}
			1 - \frac{1}{8\pi} \int_{\Sigma} H^{2} = \frac{4k}{(1+k)^{2}}.
		\end{eqnarray}
		
		\noindent{\bf Case 1: $k\in (-1,0)\cup(0,1]$.}
		We see from \eqref{geom-ineq-10} and \eqref{geom-ineq-20} that
		\begin{equation}\label{xeq-100}
			\begin{aligned}
				& \frac{(1+k)^{2}r_{0}}{\eta(r_{0})}
				\left\{
				2 \pi \frac{(1-k)^{2}}{(1+k)^{2}} - \frac{1}{4} \int_{\Sigma} H^{2}
				+ \int_{\Sigma} \left( \frac{H}{2}-\eta(r_{0})|\nabla u| \right) ^{2}
				\right\}
				\le 8\pi\left( \mathfrak{m}_{ABL}-m \right).
			\end{aligned}
		\end{equation}
		Substituting  $k$ given in \eqref{k} into \eqref{xeq-100}, we have
		\begin{equation*}
			\begin{aligned}
				&0\le \frac{(1+k)^{2}r_{0}}{\eta(r_{0})} \int_{\Sigma} \left( \frac{H}{2}-\eta(r_{0})|\nabla u| \right) ^{2}
				\le 8\pi\left( \mathfrak{m}_{ABL}-m \right).
			\end{aligned}
		\end{equation*}
		This gives \eqref{geom-ineq-4}.
		
		It follow from \eqref{geom-ineq-10} that
		\begin{equation*}
			2 \pi \geq\frac{(1+k)^{2}}{(1-k)^{2}} (\eta(r_{0}))^{2} \int_{\Sigma} |\nabla u|^2.
		\end{equation*}
		Then, by H\"older's inequality,
		\begin{equation*}
			\begin{aligned}
				2\pi \mathfrak{c}_{p}^{p-1}
				= \int_{\Sigma}|\nabla u|^{p-1}
				\le \left( \int_{\Sigma}|\nabla u|^{2} \right)^{\frac{p-1}{2}} |\Sigma|^{\frac{3-p}{2}}
				\le \left( \frac{2 \pi}{\mathfrak{c}_{p}^{-2}r_{0}^{2a}(1+k)^{4a}} \right)^{\frac{p-1}{2}} |\Sigma|^{\frac{3-p}{2}} .
			\end{aligned}
		\end{equation*}
		
		Thus
		$$
		\sqrt{\frac{|\Sigma|}{32\pi}}\geq\frac14(1+k)^{2}r_{0}=  \frac{(1+k)^2}{4|k|}(I_a(k)\mathfrak{c}_p)^{\frac1a}.
		$$
		The equality holds if and only if $M'$ is isometric to $(\mathcal{M}_{m, r_0}^{3,+},  g_{m})$.
		
		\noindent{\bf Case 2: $k=0$.} In this case, by using L'Hospital formula, let $k\rightarrow0$, we can see that $m\rightarrow0$,  $r_{0}\rightarrow2(\frac{\mathfrak{c}_p}{a})^{\frac1a}$ and $\frac{1-a\eta(r_{0})}{2k}\rightarrow\frac{a^{2}}{a+1}$.
		Also,   we may recover the procedure along this paper and we get
		$$
		\mathfrak{m}_{ABL}\ge0, \quad
		\sqrt{\frac{|\Sigma|}{32\pi}}\geq\frac{1}{2}(\frac{\mathfrak{c}_p}{a})^{\frac1a}.
		$$
		The equality holds if and only if $M'$ is isometric to $\mathbb{R}^3_+\setminus B_{+,r_{0}}^3$.
		
		\qed

\section{Proof of Theorem \ref{Thm:1.6}}\label{Sec:07}

In \cite{Agostiniani--Mantegazza--Mazzieri--Oronzio--2023},
Agostiniani--Mantegazza--Mazzieri--Oronzio gave a new proof of the
Penrose inequality for asymptotically flat manifolds with compact boundary.
In their argument, an isoperimetric inequality is used to obtain suitable
finite-perimeter representatives, while a Sobolev inequality is used to
derive the relevant estimates for the $p$-capacity. In the capillary setting,
where the ambient manifold has a noncompact boundary, one therefore needs
the corresponding Sobolev inequality on an asymptotically flat half-space.

We first recall the Euclidean input. By
\cite{Cordero--Nazaret--Villani--2004} and
\cite{Ciraolo--Figalli--Roncoroni--2020}, the sharp Sobolev inequality holds
on the half-space. In the sequel, we only need the following form.

\begin{lemma}[Sobolev inequality on the Euclidean half-space]
Let $1<p<n+1$ and set
\[
p^*:=\frac{(n+1)p}{n+1-p}.
\]
Then there exists a constant $C=C(n,p)>0$ such that for every
\[
u\in W^{1,p}(\mathbb{R}^{n+1}_+)
:=
\left\{
u\in L^{p^*}(\mathbb{R}^{n+1}_+):
\nabla u\in L^p(\mathbb{R}^{n+1}_+)
\right\},
\]
one has
\[
\left(
\int_{\mathbb{R}^{n+1}_+}|u|^{p^*}\,d\mu_\delta
\right)^{1/p^*}
\leq
C
\left(
\int_{\mathbb{R}^{n+1}_+}|\nabla u|^p\,d\mu_\delta
\right)^{1/p}.
\]
\end{lemma}

We shall also use the corresponding exterior version.

\begin{lemma}[Sobolev inequality outside a half-ball]
Let $1<p<n+1$ and set
\[
p^*:=\frac{(n+1)p}{n+1-p}.
\]
For $R>0$, define
\[
\Omega_R:=\mathbb{R}^{n+1}_+\setminus \overline{B_R(0)}.
\]
Then there exists a constant $C=C(n,p)>0$, independent of $R$, such that
for every $u\in W^{1,p}(\Omega_R)$,
\[
\left(
\int_{\Omega_R}|u|^{p^*}\,d\mu_\delta
\right)^{1/p^*}
\leq
C
\left(
\int_{\Omega_R}|\nabla u|^p\,d\mu_\delta
\right)^{1/p}.
\]
\end{lemma}

\begin{proof}
This follows from the Sobolev inequality on $\mathbb{R}^{n+1}_+$ together
with a standard extension argument for the Lipschitz domain $\Omega_R$.
The constant is independent of $R$ by scaling.
\end{proof}

We now transfer the inequality to an asymptotically flat half-space end.

\begin{lemma}[Sobolev inequality on an asymptotically flat half-space end]
Let $(M^{n+1},g)$ be an asymptotically flat half-space. Suppose that outside
a compact set $K\subset M$ there is a coordinate chart identifying
$M\setminus K$ with
\[
\mathbb{R}^{n+1}_+\setminus \overline{B_R(0)}
\]
such that
\[
g_{ij}-\delta_{ij}=O_2(r^{-\tau})
\]
for some $\tau>0$. Let $1<p<n+1$ and set
\[
p^*:=\frac{(n+1)p}{n+1-p}.
\]
Then, after enlarging $R$ if necessary, there exists a constant
$C=C(n,p,g)>0$ such that for every $u\in C_c^\infty(\overline{M\setminus K})$,
\[
\left(
\int_{M\setminus K}|u|^{p^*}\,d\mu_g
\right)^{1/p^*}
\leq
C
\left(
\int_{M\setminus K}|\nabla u|_g^p\,d\mu_g
\right)^{1/p}.
\]
Consequently, by density, the same inequality holds for every
$u\in W^{1,p}(M\setminus K,g)$.
\end{lemma}

\begin{proof}
Since $g_{ij}-\delta_{ij}=O_2(r^{-\tau})$, for every sufficiently small
$\epsilon>0$, after enlarging $R$ if necessary, we have on the end
\[
(1-\epsilon)\delta \leq g \leq (1+\epsilon)\delta
\]
as bilinear forms. Moreover, the corresponding volume forms are uniformly
comparable:
\[
(1-\epsilon)d\mu_\delta
\leq
d\mu_g
\leq
(1+\epsilon)d\mu_\delta,
\]
after possibly replacing $\epsilon$ by a comparable small constant.

Therefore,
\[
\left(
\int_{M\setminus K}|u|^{p^*}\,d\mu_g
\right)^{1/p^*}
\leq
(1+\epsilon)^{1/p^*}
\left(
\int_{\Omega_R}|u|^{p^*}\,d\mu_\delta
\right)^{1/p^*}.
\]
By the Euclidean Sobolev inequality on $\Omega_R$,
\[
\left(
\int_{\Omega_R}|u|^{p^*}\,d\mu_\delta
\right)^{1/p^*}
\leq
C(n,p)
\left(
\int_{\Omega_R}|\nabla u|_\delta^p\,d\mu_\delta
\right)^{1/p}.
\]
Using again the uniform comparability between $g$ and $\delta$, we get
\[
\left(
\int_{\Omega_R}|\nabla u|_\delta^p\,d\mu_\delta
\right)^{1/p}
\leq
C(\epsilon)
\left(
\int_{M\setminus K}|\nabla u|_g^p\,d\mu_g
\right)^{1/p}.
\]
Combining the preceding estimates gives the desired inequality.
\end{proof}

Finally, we need the following version of the gluing result of
Pigola--Setti--Troyanov. The original statement in
\cite{Pigola--Setti--Troyanov--2014, Theorem 3.2} is formulated for
possibly incomplete Riemannian manifolds without boundary. The same proof
applies verbatim to manifolds with boundary, provided one works with
smooth functions up to the boundary and with the corresponding
free-boundary Sobolev spaces.

\begin{lemma}[$L^{q,p}$-Sobolev inequality outside a compact set]
\label{lem:SobolevInequality}
Let $(M,g)$ be a possibly incomplete Riemannian manifold with boundary and
infinite volume. Suppose that there exists a compact subset $K\subset M$
such that $M\setminus K$ supports the Sobolev inequality
\[
\left(
\int_{M\setminus K}|\varphi|^q\,d\mu_g
\right)^{1/q}
\leq
C_K
\left(
\int_{M\setminus K}|\nabla \varphi|_g^p\,d\mu_g
\right)^{1/p}
\]
for every
\[
\varphi\in C_c^\infty(\overline{M\setminus K}).
\]
Then there exists a constant $C_M>0$ such that
\[
\left(
\int_M|\varphi|^q\,d\mu_g
\right)^{1/q}
\leq
C_M
\left(
\int_M|\nabla \varphi|_g^p\,d\mu_g
\right)^{1/p}
\]
for every
\[
\varphi\in C_c^\infty(\overline M).
\]
\end{lemma}
\begin{proof}
The argument is the same as in
\cite{Pigola--Setti--Troyanov--2014, Theorem 3.2}. We recall the proof for
completeness. Choose a precompact domain $\Omega\subset M$ with smooth
boundary, allowing boundary along $\partial M$, such that $K\subset\Omega$.
Let $\Omega_\varepsilon$ be a collar enlargement of $\Omega$ and set
\[
M_\varepsilon:=M\setminus \Omega_\varepsilon .
\]
By assumption, the $L^{q,p}$-Sobolev inequality holds on $M_\varepsilon$.
On the compact manifold with boundary $\Omega_\varepsilon$, the standard
Sobolev inequality holds for functions smooth up to the boundary:
\[
\|w\|_{L^q(\Omega_\varepsilon)}
\leq
C_\varepsilon
\left(
\|\nabla w\|_{L^p(\Omega_\varepsilon)}
+
\|w\|_{L^p(\Omega_\varepsilon)}
\right).
\]
Choose a smooth cut-off function $\rho\in C_c^\infty(\overline M)$ such that
$0\leq \rho\leq 1$, $\rho=1$ on a smaller neighborhood of $\Omega$, and
$\rho=0$ on $M_\varepsilon$. For any
$v\in C_c^\infty(\overline M)$, write
\[
v=\rho v+(1-\rho)v.
\]
Then $\rho v$ is supported in the compact region $\Omega_\varepsilon$, while
$(1-\rho)v$ is supported in the exterior region $M_\varepsilon$. Applying the
compact Sobolev inequality to $\rho v$ and the exterior Sobolev inequality
to $(1-\rho)v$, and using the bounds on $\nabla\rho$, gives
\[
\|v\|_{L^q(M)}
\leq
C
\left(
\|\nabla v\|_{L^p(M)}
+
\|v\|_{L^p(\Omega_\varepsilon)}
\right).
\]
Since $M$ has infinite volume and the exterior Sobolev inequality holds,
the same argument as in
\cite{Pigola--Setti--Troyanov--2014, Theorem 3.2} implies that $M$ is
$p$-hyperbolic. Hence the local $L^p$ term can be absorbed by the
$p$-hyperbolic Poincaré-type inequality on $\Omega_\varepsilon$:
\[
\|v\|_{L^p(\Omega_\varepsilon)}
\leq
C
\|\nabla v\|_{L^p(M)}.
\]
Therefore
\[
\|v\|_{L^q(M)}
\leq
C_M
\|\nabla v\|_{L^p(M)}
\]
for every $v\in C_c^\infty(\overline M)$.
\end{proof}

Combining the previous lemmas, we obtain the global Sobolev inequality on
an asymptotically flat half-space.

\begin{corollary}[Global Sobolev inequality on an asymptotically flat half-space]
\label{cor:SobolevInequality}
Let $(M^{n+1},g)$ be an asymptotically flat half-space with one end.
Let $1<p<n+1$ and set
\[
p^*:=\frac{(n+1)p}{n+1-p}.
\]
Then there exists a constant $C=C(M,g,p)>0$ such that for every
$\varphi\in C_c^\infty(\overline{M})$,
\[
\left(
\int_M|\varphi|^{p^*}\,d\mu_g
\right)^{1/p^*}
\leq
C
\left(
\int_M|\nabla \varphi|_g^p\,d\mu_g
\right)^{1/p}.
\]
In particular,
\[
\|\varphi\|_{L^{p^*}(M,g)}
\leq
C
\|\nabla \varphi\|_{L^p(M,g)}.
\]
Moreover, in dimension three, the constants may be chosen uniformly bounded
for $p\in(1,p_0]$, where $p_0>1$ is fixed sufficiently close to $1$.
\end{corollary}

We now apply this Sobolev inequality to the $p$-capacity of the compact
free-boundary component $\Sigma$.

\begin{theorem}[Limit of the $p$-capacity as $p\to 1$]
\label{thm:capacity-limit}
Let $(M',g)$ be a $3$-dimensional, complete, one-ended asymptotically flat
half-space with boundary
\[
\partial M'=\hat{\partial} M' \cup \Sigma,
\]
where $\Sigma$ is a compact connected free-boundary minimal surface. Assume
that $M'$ has nonnegative scalar curvature, that
\[
H_{\hat{\partial} M' }\geq 0,
\]
that $M'$ is simply connected, and that $M'$ contains no other closed or
free-boundary minimal surfaces. For $1<p<3$, let
\[
\operatorname{Cap}_p(\Sigma,M')
:=
\inf
\left\{
\int_{M'}|\nabla \varphi|_g^p\,d\mu_g:
\varphi\in C_c^\infty(\overline{M'}),\ 
\varphi\geq 1 \ \text{in a neighborhood of } \Sigma
\right\}.
\]
Then
\[
\lim_{p\to 1^+}\operatorname{Cap}_p(\Sigma,M')
=
|\Sigma|_g .
\]
Equivalently, if $u_p$ is the weak solution of the $p$-capacitary problem
\eqref{p-H}, then
\[
\lim_{p\to 1^+}
\int_{M'}|\nabla u_p|_g^p\,d\mu_g
=
|\Sigma|_g .
\]
\end{theorem}

\begin{proof}
We divide the proof into the lower and upper bounds.

\medskip

\noindent
\emph{Step 1: The lower bound.}
Define the relative $1$-capacity by
\[
\operatorname{Cap}_1(\Sigma,M')
:=
\inf
\left\{
\int_{M'}|\nabla \varphi|_g\,d\mu_g:
\varphi\in C_c^\infty(\overline{M'}),\
\varphi\geq 1 \ \text{in a neighborhood of } \Sigma
\right\}.
\]
We first claim that
\[
\operatorname{Cap}_1(\Sigma,M')\geq |\Sigma|_g .
\]

Indeed, let $\varphi$ be admissible for $\operatorname{Cap}_1(\Sigma,M')$.
After truncation, we may assume that
\[
0\leq \varphi\leq 1.
\]
For a.e. $t\in(0,1)$, the superlevel set
\[
E_t:=\{\varphi>t\}
\]
has finite perimeter, contains a neighborhood of $\Sigma$, and separates
$\Sigma$ from the asymptotically flat end. Its relative reduced boundary
$\partial^*E_t$ is therefore an admissible surface enclosing $\Sigma$, with
free boundary allowed along $\hat{\partial} M' $.

We use the following convention for relative perimeter. If $E\subset M'$ is
a set of locally finite perimeter, we define
\[
P(E;M')
:=
|D\chi_E|_g(\operatorname{int}M').
\]
Equivalently, when $E$ has smooth boundary, 
\[
P(E;M')
=
\mathcal H^2_g\bigl(\partial E\cap \operatorname{int}M'\bigr).
\]
Thus the portion of $\partial E$ lying on the noncompact boundary
$\hat{\partial} M' $ is not counted. This is the natural perimeter in the
free-boundary setting.

Let $\mathcal A_\Sigma$ denote the class of bounded finite-perimeter sets
$E\subset M'$ such that $E$ contains a one-sided neighborhood of $\Sigma$ and
$M'\setminus E$ contains the asymptotically flat end. We say that $\Sigma$ is
relative outer-minimizing if
\[
P(E;M')\geq |\Sigma|_g
\]
for every $E\in\mathcal A_\Sigma$.

Under the assumptions of Theorem \ref{Thm:1.6}, $\Sigma$ is relative
outer-minimizing. Indeed, suppose otherwise. Then there exists
$E_0\in\mathcal A_\Sigma$ such that
\[
P(E_0;M')<|\Sigma|_g.
\]
Consider the least-area relative enclosure of $\Sigma$, namely a minimizer of
\[
\inf\bigl\{P(E;M'):\ E\in\mathcal A_\Sigma\bigr\}.
\]
By the compactness and lower semicontinuity of sets of finite perimeter, and
by the standard regularity theory for relative perimeter minimizers in
dimension three, such a minimizer has smooth reduced boundary. Moreover, its
reduced boundary is minimal in $\operatorname{int}M'$ and satisfies the
free-boundary condition along $\hat{\partial} M' $.

Since the infimum is strictly less than $|\Sigma|_g$, this minimizing
relative boundary cannot be equal to $\Sigma$. Hence it produces a closed
minimal surface or a free-boundary minimal surface different from $\Sigma$.
This contradicts the assumption that $M'$ contains no other closed or
free-boundary minimal surfaces. Therefore $\Sigma$ must be relative
outer-minimizing.
% By the assumptions above, the standard minimizing-hull argument implies that
% $\Sigma$ is relative outer-minimizing. Hence, for a.e. $t\in(0,1)$,
% \[
% P(E_t;M')\geq |\Sigma|_g,
% \]
% where $P(E_t;M')$ denotes the relative perimeter in $M'$. 

Now let $\varphi$ be admissible for $\operatorname{Cap}_1(\Sigma,M')$.
After truncating, we may assume $0\leq \varphi\leq 1$. For a.e.
$t\in(0,1)$, the superlevel set
\[
E_t:=\{\varphi>t\}
\]
has finite perimeter. Since $\varphi=1$ in a neighborhood of $\Sigma$ and
$\varphi$ has compact support, $E_t$ contains a one-sided neighborhood of
$\Sigma$ and is bounded. Hence
\[
E_t\in \mathcal A_\Sigma.
\]
By the relative outer-minimizing property of $\Sigma$, we therefore have
\[
P(E_t;M')\geq |\Sigma|_g
\]
for a.e. $t\in(0,1)$.

% By the coarea formula,
% \[
% \int_{M'}|\nabla \varphi|_g\,d\mu_g
% =
% \int_0^1 P(E_t;M')\,dt.
% \]
Let $\mathring M'=M'\setminus \partial M'$. For
$\varphi\in C_c^\infty(\overline{M'})$ and
$E_t:=\{\varphi>t\}$, the coarea formula gives
\[
\int_{M'}|\nabla \varphi|_g\,d\mu_g
=
\int_{\mathring M'}|\nabla \varphi|_g\,d\mu_g
=
\int_{-\infty}^{\infty} P(E_t;M')\,dt.
\]
In particular, if $0\leq \varphi\leq 1$, then
\[
\int_{M'}|\nabla \varphi|_g\,d\mu_g
=
\int_0^1 P(E_t;M')\,dt,
\]
where
\[
P(E_t;M')
:=
|D\chi_{E_t}|_g(\mathring M')
=
\mathcal H^2_g(\partial^*E_t\cap \mathring M')
\]
is the relative perimeter of $E_t$ in $M'$. In particular, the portion of
$\partial^*E_t$ lying on the noncompact boundary $\hat{\partial} M' $ is
not counted.

Since $E_t$ contains a one-sided neighborhood of $\Sigma$ and separates
$\Sigma$ from the asymptotically flat end, $E_t$ is an admissible relative
enclosing set for $\Sigma$. By the relative outer-minimizing property of
$\Sigma$,
\[
P(E_t;M')\geq |\Sigma|_g
\]
for a.e. $t\in(0,1)$.

Using the preceding inequality, we obtain
\[
\int_{M'}|\nabla \varphi|_g\,d\mu_g
\geq
\int_0^1 |\Sigma|_g\,dt
=
|\Sigma|_g.
\]
Taking the infimum over all admissible $\varphi$ gives
\[
\operatorname{Cap}_1(\Sigma,M')
\geq
|\Sigma|_g.
\]

We next compare $\operatorname{Cap}_1(\Sigma,M')$ with the limit of
$\operatorname{Cap}_p(\Sigma,M')$ as $p\to1^+$. Let $u_p$ be the
$p$-capacitary potential. By truncation, we may assume
\[
0\leq u_p\leq 1.
\]
Moreover,
\[
\operatorname{Cap}_p(\Sigma,M')
=
\int_{M'}|\nabla u_p|_g^p\,d\mu_g.
\]

By Corollary \ref{cor:SobolevInequality}, for $p>1$ sufficiently close to
$1$,
\[
\left(
\int_{M'} |f|^{\frac{3p}{3-p}}\,d\mu_g
\right)^{\frac{3-p}{3p}}
\leq
C_S(p)
\left(
\int_{M'}|\nabla f|_g^p\,d\mu_g
\right)^{1/p},
\]
where the constants $C_S(p)$ are uniformly bounded for $p\in(1,p_0]$.

Set
\[
q=q(p):=\frac{2p}{3-p}.
\]
Then
\[
q\to1
\qquad\text{as }p\to1^+,
\]
and
\[
\frac{(q-1)p}{p-1}
=
\frac{3p}{3-p}.
\]
Using $u_p^q$ as a competitor for $\operatorname{Cap}_1(\Sigma,M')$
or, equivalently, arguing by a standard cutoff approximation at infinity, we
obtain
\[
\operatorname{Cap}_1(\Sigma,M')
\leq
\int_{M'}|\nabla(u_p^q)|_g\,d\mu_g.
\]
By H\"older's inequality,
\[
\int_{M'}|\nabla(u_p^q)|_g\,d\mu_g
=
q\int_{M'}u_p^{q-1}|\nabla u_p|_g\,d\mu_g
\]
\[
\leq
q
\left(
\int_{M'}u_p^{\frac{(q-1)p}{p-1}}\,d\mu_g
\right)^{\frac{p-1}{p}}
\left(
\int_{M'}|\nabla u_p|_g^p\,d\mu_g
\right)^{1/p}.
\]
By the choice of $q$,
\[
\frac{(q-1)p}{p-1}=\frac{3p}{3-p}.
\]
Thus
\[
\operatorname{Cap}_1(\Sigma,M')
\leq
q
\left(
\int_{M'}u_p^{\frac{3p}{3-p}}\,d\mu_g
\right)^{\frac{p-1}{p}}
\operatorname{Cap}_p(\Sigma,M')^{1/p}.
\]
The Sobolev inequality applied to $u_p$ gives
\[
\left(
\int_{M'}u_p^{\frac{3p}{3-p}}\,d\mu_g
\right)^{\frac{3-p}{3p}}
\leq
C_S(p)
\operatorname{Cap}_p(\Sigma,M')^{1/p}.
\]
Since
\[
\frac{p-1}{p}
=
(q-1)\frac{3-p}{3p},
\]
we get
\[
\left(
\int_{M'}u_p^{\frac{3p}{3-p}}\,d\mu_g
\right)^{\frac{p-1}{p}}
\leq
C_S(p)^{q-1}
\operatorname{Cap}_p(\Sigma,M')^{\frac{q-1}{p}}.
\]
Therefore,
\[
\operatorname{Cap}_1(\Sigma,M')
\leq
q\,C_S(p)^{q-1}
\operatorname{Cap}_p(\Sigma,M')^{q/p}.
\]
Since
\[
q\to1,
\qquad
\frac{q}{p}=\frac{2}{3-p}\to1,
\qquad
C_S(p)^{q-1}\to1,
\]
and since the capacities are uniformly bounded above near $p=1$ by testing
against any fixed admissible cutoff function, we obtain
\[
\operatorname{Cap}_1(\Sigma,M')
\leq
\liminf_{p\to1^+}\operatorname{Cap}_p(\Sigma,M').
\]
Together with $\operatorname{Cap}_1(\Sigma,M')\geq |\Sigma|_g$, this yields
\[
\liminf_{p\to1^+}\operatorname{Cap}_p(\Sigma,M')
\geq
|\Sigma|_g .
\]

\medskip

\noindent
\emph{Step 2: The upper bound.}
We construct admissible competitors concentrating near $\Sigma$. Since
$\Sigma$ is smooth, compact, and meets $\hat{\partial} M' $ orthogonally,
the free-boundary tubular neighborhood theorem gives an $\varepsilon_0>0$
such that the distance function
\[
\rho(x):=\operatorname{dist}_g(x,\Sigma)
\]
is smooth in the one-sided collar neighborhood
\[
\mathcal U_{\varepsilon_0}
:=
\{x\in M':0\leq \rho(x)<\varepsilon_0\}.
\]
Moreover,
\[
|\nabla \rho|_g=1
\]
in $\mathcal U_{\varepsilon_0}$.

Fix $0<\varepsilon<\varepsilon_0/2$ and choose a smooth nonincreasing cutoff
function $\chi_\varepsilon:[0,\infty)\to[0,1]$ such that
\[
\chi_\varepsilon(t)=1
\quad\text{for }0\leq t\leq \varepsilon,
\]
\[
\chi_\varepsilon(t)=0
\quad\text{for }t\geq 2\varepsilon.
\]
Then
\[
\int_\varepsilon^{2\varepsilon}
|\chi_\varepsilon'(t)|\,dt
=
1.
\]
Define
\[
\eta_\varepsilon(x):=\chi_\varepsilon(\rho(x)).
\]
Then $\eta_\varepsilon$ is admissible for $\operatorname{Cap}_p(\Sigma,M')$.
Therefore,
\[
\operatorname{Cap}_p(\Sigma,M')
\leq
\int_{M'}|\nabla \eta_\varepsilon|_g^p\,d\mu_g.
\]
Since $|\nabla \rho|_g=1$ in the collar, we have
\[
|\nabla \eta_\varepsilon|_g
=
|\chi_\varepsilon'(\rho)|.
\]
By the coarea formula,
\[
\int_{M'}|\nabla \eta_\varepsilon|_g^p\,d\mu_g
=
\int_\varepsilon^{2\varepsilon}
|\chi_\varepsilon'(t)|^p |\Sigma_t|_g\,dt,
\]
where
\[
\Sigma_t:=\{\rho=t\}.
\]
Letting $p\to1^+$ while keeping $\varepsilon$ fixed gives
\[
\limsup_{p\to1^+}\operatorname{Cap}_p(\Sigma,M')
\leq
\int_\varepsilon^{2\varepsilon}
|\chi_\varepsilon'(t)|\,|\Sigma_t|_g\,dt.
\]
Since
\[
\int_\varepsilon^{2\varepsilon}
|\chi_\varepsilon'(t)|\,dt
=
1,
\]
we obtain
\[
\limsup_{p\to1^+}\operatorname{Cap}_p(\Sigma,M')
\leq
\sup_{t\in(\varepsilon,2\varepsilon)}|\Sigma_t|_g.
\]
Finally, the parallel free-boundary surfaces $\Sigma_t$ converge smoothly to
$\Sigma$ as $t\to0^+$. Hence
\[
|\Sigma_t|_g\to|\Sigma|_g,
\]
and letting $\varepsilon\to0^+$ yields
\[
\limsup_{p\to1^+}\operatorname{Cap}_p(\Sigma,M')
\leq
|\Sigma|_g .
\]

Combining the lower and upper bounds, we conclude that
\[
\lim_{p\to1^+}\operatorname{Cap}_p(\Sigma,M')
=
|\Sigma|_g .
\]
The final identity involving the $p$-capacitary potential $u_p$ follows from
the variational characterization of $u_p$.
\end{proof}

\noindent{\it Proof of Theorem \ref{Thm:1.6}.}
On the one hand, we have
\[\lim_{p\to1}I_a(1)^{\frac1a}=\frac14\] and 
\[\lim_{p\to1^+}\mathfrak{c}_p^{\frac1a}=\sqrt{\frac{{\rm Cap}_1(\Sigma,M')}{2\pi}}\]

On the other hand, since $\Sigma$ is a unique  free boundary minimal surface and $M'$ contains no other closed or free boundary minimal surface, by using Theorem \ref{thm:capacity-limit}, we obtain that $$\lim_{p\to1}{\rm Cap}(\Sigma,M')=|\Sigma|.$$

Thus, it is easy to see that 
$$\mathfrak{m}_{ABL}\geq\sqrt{\frac{|\Sigma|}{32\pi}}.$$ 

This completes the proof of Theorem \ref{Thm:1.6}.

\vskip 2mm
{\bf Acknowledgement.} J. Yin would like to thank Professor Li Chen  for his helpful discussions. C. Xia is supported by NSFC (Grant No. 12271449, 12526203, 12526102) and the Natural Science Foundation of Fujian Province of China (Grant No. 2024J011008).
J. Yin is supported by the NSFC (Grant No. 12201138) and Mathematics Tianyuan fund project (Grant No. 12226350) and  NSF of Henan Province (Grant No. 262300421869).

\printbibliography

@article {Agostiniani--Mantegazza--Mazzieri--Oronzio--2023,
    AUTHOR = {Agostiniani, Virginia and Mantegazza, Carlo and Mazzieri,
              Lorenzo and Oronzio, Francesca},
     TITLE = {A new proof of the {R}iemannian {P}enrose inequality},
   JOURNAL = {Atti Accad. Naz. Lincei Rend. Lincei Mat. Appl.},
  FJOURNAL = {Atti della Accademia Nazionale dei Lincei. Rendiconti Lincei.
              Matematica e Applicazioni},
    VOLUME = {34},
      YEAR = {2023},
    NUMBER = {3},
     PAGES = {715--726},
      ISSN = {1120-6330,1720-0768},
   MRCLASS = {53C21 (31C12 31C15 53Z05)},
  MRNUMBER = {4681223},
MRREVIEWER = {Changwei\ Xiong},
       DOI = {10.4171/rlm/1024},
       URL = {https://doi.org/10.4171/rlm/1024},
}

@article {Agostiniani--Mazzieri--Oronzio--2022,
    AUTHOR = {Agostiniani, Virginia and Mazzieri, Lorenzo and Oronzio,
              Francesca},
     TITLE = {A geometric capacitary inequality for sub-static manifolds
              with harmonic potentials},
   JOURNAL = {Math. Eng.},
  FJOURNAL = {Mathematics in Engineering},
    VOLUME = {4},
      YEAR = {2022},
    NUMBER = {2},
     PAGES = {Paper No. 013, 40},
      ISSN = {2640-3501},
   MRCLASS = {53C20 (31C12 35C20 53C24 58J05)},
  MRNUMBER = {4281178},
MRREVIEWER = {Stefano\ Borghini},
       DOI = {10.3934/mine.2022013},
       URL = {https://doi.org/10.3934/mine.2022013},
}

@article {Almaraz--Barbosa--deLima--2016,
    AUTHOR = {Almaraz, S\'ergio and Barbosa, Ezequiel and de Lima, Levi
              Lopes},
     TITLE = {A positive mass theorem for asymptotically flat manifolds with
              a non-compact boundary},
   JOURNAL = {Comm. Anal. Geom.},
  FJOURNAL = {Communications in Analysis and Geometry},
    VOLUME = {24},
      YEAR = {2016},
    NUMBER = {4},
     PAGES = {673--715},
      ISSN = {1019-8385,1944-9992},
   MRCLASS = {53C20 (53C21)},
  MRNUMBER = {3570413},
MRREVIEWER = {Mohameden\ Ahmedou},
       DOI = {10.4310/CAG.2016.v24.n4.a1},
       URL = {https://doi.org/10.4310/CAG.2016.v24.n4.a1},
}

@article {Almaraz--deLima--2023,
    AUTHOR = {Almaraz, S\'ergio and de Lima, Levi Lopes},
     TITLE = {Mass, center of mass and isoperimetry in asymptotically flat
              3-manifolds},
   JOURNAL = {Calc. Var. Partial Differential Equations},
  FJOURNAL = {Calculus of Variations and Partial Differential Equations},
    VOLUME = {62},
      YEAR = {2023},
    NUMBER = {7},
     PAGES = {Paper No. 196, 41},
      ISSN = {0944-2669,1432-0835},
   MRCLASS = {53C21 (53A10 83C05)},
  MRNUMBER = {4616637},
MRREVIEWER = {Yuguang\ Shi},
       DOI = {10.1007/s00526-023-02519-1},
       URL = {https://doi.org/10.1007/s00526-023-02519-1},
}

@article {Benatti--Fogagnolo--Mazzieri--2024--MathAnn,
    AUTHOR = {Benatti, Luca and Fogagnolo, Mattia and Mazzieri, Lorenzo},
     TITLE = {The asymptotic behaviour of {$p$}-capacitary potentials in
              asymptotically conical manifolds},
   JOURNAL = {Math. Ann.},
  FJOURNAL = {Mathematische Annalen},
    VOLUME = {388},
      YEAR = {2024},
    NUMBER = {1},
     PAGES = {99--139},
      ISSN = {0025-5831,1432-1807},
   MRCLASS = {58K55 (31C12 53E10)},
  MRNUMBER = {4693930},
MRREVIEWER = {Nicola\ Gigli},
       DOI = {10.1007/s00208-022-02515-4},
       URL = {https://doi.org/10.1007/s00208-022-02515-4},
}

@article {Bray--2001,
    AUTHOR = {Bray, Hubert L.},
     TITLE = {Proof of the {R}iemannian {P}enrose inequality using the
              positive mass theorem},
   JOURNAL = {J. Differential Geom.},
  FJOURNAL = {Journal of Differential Geometry},
    VOLUME = {59},
      YEAR = {2001},
    NUMBER = {2},
     PAGES = {177--267},
      ISSN = {0022-040X,1945-743X},
   MRCLASS = {53C21 (53C44 83C57 83C75)},
  MRNUMBER = {1908823},
MRREVIEWER = {John\ Urbas},
       URL = {http://projecteuclid.org/euclid.jdg/1090349428},
}

@article {Bray--Kazaras--Khuri--Stern--2022,
    AUTHOR = {Bray, Hubert L. and Kazaras, Demetre P. and Khuri, Marcus A.
              and Stern, Daniel L.},
     TITLE = {Harmonic functions and the mass of 3-dimensional
              asymptotically flat {R}iemannian manifolds},
   JOURNAL = {J. Geom. Anal.},
  FJOURNAL = {Journal of Geometric Analysis},
    VOLUME = {32},
      YEAR = {2022},
    NUMBER = {6},
     PAGES = {Paper No. 184, 29},
      ISSN = {1050-6926,1559-002X},
   MRCLASS = {53C20 (53C21 58J90 58Z05 83C40 83C57)},
  MRNUMBER = {4411747},
MRREVIEWER = {Alfonso\ Garc\'ia-Parrado},
       DOI = {10.1007/s12220-022-00924-0},
       URL = {https://doi.org/10.1007/s12220-022-00924-0},
}

@article {Bray--Lee--2009,
    AUTHOR = {Bray, Hubert L. and Lee, Dan A.},
     TITLE = {On the {R}iemannian {P}enrose inequality in dimensions less
              than eight},
   JOURNAL = {Duke Math. J.},
  FJOURNAL = {Duke Mathematical Journal},
    VOLUME = {148},
      YEAR = {2009},
    NUMBER = {1},
     PAGES = {81--106},
      ISSN = {0012-7094,1547-7398},
   MRCLASS = {53C21 (53C44 53C80)},
  MRNUMBER = {2515101},
MRREVIEWER = {Gabjin\ Yun},
       DOI = {10.1215/00127094-2009-020},
       URL = {https://doi.org/10.1215/00127094-2009-020},
}

@article {Bray--Miao--2008,
    AUTHOR = {Bray, Hubert and Miao, Pengzi},
     TITLE = {On the capacity of surfaces in manifolds with nonnegative
              scalar curvature},
   JOURNAL = {Invent. Math.},
  FJOURNAL = {Inventiones Mathematicae},
    VOLUME = {172},
      YEAR = {2008},
    NUMBER = {3},
     PAGES = {459--475},
      ISSN = {0020-9910,1432-1297},
   MRCLASS = {53C21 (53C20)},
  MRNUMBER = {2393076},
MRREVIEWER = {Gabjin\ Yun},
       DOI = {10.1007/s00222-007-0102-x},
       URL = {https://doi.org/10.1007/s00222-007-0102-x},
}

@article {Chan--Chu--Lee--Tsang--2024,
    AUTHOR = {Chan, Pak-Yeung and Chu, Jianchun and Lee, Man-Chun and Tsang,
              Tin-Yau},
     TITLE = {Monotonicity of the {$p$}-{G}reen functions},
   JOURNAL = {Int. Math. Res. Not. IMRN},
  FJOURNAL = {International Mathematics Research Notices. IMRN},
      YEAR = {2024},
    NUMBER = {9},
     PAGES = {7998--8025},
      ISSN = {1073-7928,1687-0247},
   MRCLASS = {53C21},
  MRNUMBER = {4742853},
MRREVIEWER = {Mattia\ Fogagnolo},
       DOI = {10.1093/imrn/rnae030},
       URL = {https://doi.org/10.1093/imrn/rnae030},
}

@article {Chodosh--Li--2024,
    AUTHOR = {Chodosh, Otis and Li, Chao},
     TITLE = {Stable minimal hypersurfaces in {${\bf R}^4$}},
   JOURNAL = {Acta Math.},
  FJOURNAL = {Acta Mathematica},
    VOLUME = {233},
      YEAR = {2024},
    NUMBER = {1},
     PAGES = {1--31},
      ISSN = {0001-5962,1871-2509},
   MRCLASS = {53A10 (53C42)},
  MRNUMBER = {4816633},
MRREVIEWER = {Jinyu\ Guo},
       DOI = {10.4310/acta.2024.v233.n1.a1},
       URL = {https://doi.org/10.4310/acta.2024.v233.n1.a1},
}

@article {Chrusciel--1990,
    AUTHOR = {Chru\'sciel, Piotr T.},
     TITLE = {Asymptotic estimates in weighted {H}\"older spaces for a class
              of elliptic scale-covariant second order operators},
   JOURNAL = {Ann. Fac. Sci. Toulouse Math. (5)},
  FJOURNAL = {Toulouse. Facult\'e{} des Sciences. Annales. Math\'ematiques.
              S\'erie 5},
    VOLUME = {11},
      YEAR = {1990},
    NUMBER = {1},
     PAGES = {21--37},
      ISSN = {0240-2955},
   MRCLASS = {35B45 (35J15 35P20 47F05)},
  MRNUMBER = {1191470},
       URL = {http://www.numdam.org/item?id=AFST_1990_5_11_1_21_0},
}

@article {Ciraolo--Figalli--Roncoroni--2020,
    AUTHOR = {Ciraolo, Giulio and Figalli, Alessio and Roncoroni, Alberto},
     TITLE = {Symmetry results for critical anisotropic {$p$}-{L}aplacian
              equations in convex cones},
   JOURNAL = {Geom. Funct. Anal.},
  FJOURNAL = {Geometric and Functional Analysis},
    VOLUME = {30},
      YEAR = {2020},
    NUMBER = {3},
     PAGES = {770--803},
      ISSN = {1016-443X,1420-8970},
   MRCLASS = {35J92 (35B06 35B33)},
  MRNUMBER = {4135671},
MRREVIEWER = {Ky\ Ho},
       DOI = {10.1007/s00039-020-00535-3},
       URL = {https://doi.org/10.1007/s00039-020-00535-3},
}

@article {Cordero--Nazaret--Villani--2004,
    AUTHOR = {Cordero-Erausquin, D. and Nazaret, B. and Villani, C.},
     TITLE = {A mass-transportation approach to sharp {S}obolev and
              {G}agliardo-{N}irenberg inequalities},
   JOURNAL = {Adv. Math.},
  FJOURNAL = {Advances in Mathematics},
    VOLUME = {182},
      YEAR = {2004},
    NUMBER = {2},
     PAGES = {307--332},
      ISSN = {0001-8708,1090-2082},
   MRCLASS = {26D15 (46E35)},
  MRNUMBER = {2032031},
MRREVIEWER = {Olivier\ Druet},
       DOI = {10.1016/S0001-8708(03)00080-X},
       URL = {https://doi.org/10.1016/S0001-8708(03)00080-X},
}

@article{DiBenedetto--1983--NATMA,
title = {C1 + α local regularity of weak solutions of degenerate elliptic equations},
journal = {Nonlinear Analysis: Theory, Methods \& Applications},
volume = {7},
number = {8},
pages = {827-850},
year = {1983},
issn = {0362-546X},
doi = {https://doi.org/10.1016/0362-546X(83)90061-5},
url = {https://www.sciencedirect.com/science/article/pii/0362546X83900615},
author = {E. DiBenedetto}
}

@article{DiBenedetio--1983--PAPP,
  title={Interior and boundary regularity for a class of free boundary problems},
  author={BENEDETIO, E Dl},
  journal={Free boundary problems: theory and applications},
  volume={78},
  pages={383},
  year={1983},
  publisher={Pitman Advanced Pub. Program}
}

@article {Eichmair--Koerber--2023,
    AUTHOR = {Eichmair, Michael and Koerber, Thomas},
     TITLE = {Doubling of asymptotically flat half-spaces and the
              {R}iemannian {P}enrose inequality},
   JOURNAL = {Comm. Math. Phys.},
  FJOURNAL = {Communications in Mathematical Physics},
    VOLUME = {400},
      YEAR = {2023},
    NUMBER = {3},
     PAGES = {1823--1860},
      ISSN = {0010-3616,1432-0916},
   MRCLASS = {53C21 (53C20)},
  MRNUMBER = {4595610},
MRREVIEWER = {Hans-Bert\ Rademacher},
       DOI = {10.1007/s00220-023-04635-7},
       URL = {https://doi.org/10.1007/s00220-023-04635-7},
}

@article {Hirsch--Miao--Tam--2024,
    AUTHOR = {Hirsch, Sven and Miao, Pengzi and Tam, Luen-Fai},
     TITLE = {Monotone quantities of {$p$}-harmonic functions and their
              applications},
   JOURNAL = {Pure Appl. Math. Q.},
  FJOURNAL = {Pure and Applied Mathematics Quarterly},
    VOLUME = {20},
      YEAR = {2024},
    NUMBER = {2},
     PAGES = {599--644},
      ISSN = {1558-8599,1558-8602},
   MRCLASS = {53C20},
  MRNUMBER = {4734880},
MRREVIEWER = {Jui-Tang\ Chen},
       DOI = {10.4310/pamq.2024.v20.n2.a1},
       URL = {https://doi.org/10.4310/pamq.2024.v20.n2.a1},
}

@article {Huisken--Ilmanen--2001,
    AUTHOR = {Huisken, Gerhard and Ilmanen, Tom},
     TITLE = {The inverse mean curvature flow and the {R}iemannian {P}enrose
              inequality},
   JOURNAL = {J. Differential Geom.},
  FJOURNAL = {Journal of Differential Geometry},
    VOLUME = {59},
      YEAR = {2001},
    NUMBER = {3},
     PAGES = {353--437},
      ISSN = {0022-040X,1945-743X},
   MRCLASS = {53C44 (35D05 35J20 35J60 53C21 53C42 83C57)},
  MRNUMBER = {1916951},
MRREVIEWER = {John\ Urbas},
       URL = {http://projecteuclid.org/euclid.jdg/1090349447},
}

@article {Jia--Wang--Xia--Zhang--2024,
    AUTHOR = {Jia, Xiaohan and Wang, Guofang and Xia, Chao and Zhang, Xuwen},
     TITLE = {Heintze-{K}archer inequality for anisotropic free boundary
              hypersurfaces in convex domains},
   JOURNAL = {J. Math. Study},
  FJOURNAL = {Journal of Mathematical Study. Shuxue Yanjiu},
    VOLUME = {57},
      YEAR = {2024},
    NUMBER = {3},
     PAGES = {243--258},
      ISSN = {2096-9856,2617-8702},
   MRCLASS = {53C42 (53A10 53C45)},
  MRNUMBER = {4817296},
MRREVIEWER = {Jinyu\ Guo},
}

@article {Koerber--2023,
    AUTHOR = {Koerber, Thomas},
     TITLE = {The {R}iemannian {P}enrose inequality for asymptotically flat
              manifolds with non-compact boundary},
   JOURNAL = {J. Differential Geom.},
  FJOURNAL = {Journal of Differential Geometry},
    VOLUME = {124},
      YEAR = {2023},
    NUMBER = {2},
     PAGES = {317--379},
      ISSN = {0022-040X,1945-743X},
   MRCLASS = {53C20 (53C80)},
  MRNUMBER = {4602727},
MRREVIEWER = {Dan\ A.\ Lee},
       DOI = {10.4310/jdg/1686931603},
       URL = {https://doi.org/10.4310/jdg/1686931603},
}

@article {Lieberman--1988,
    AUTHOR = {Lieberman, Gary M.},
     TITLE = {Boundary regularity for solutions of degenerate elliptic
              equations},
   JOURNAL = {Nonlinear Anal.},
  FJOURNAL = {Nonlinear Analysis. Theory, Methods \& Applications. An
              International Multidisciplinary Journal},
    VOLUME = {12},
      YEAR = {1988},
    NUMBER = {11},
     PAGES = {1203--1219},
      ISSN = {0362-546X,1873-5215},
   MRCLASS = {35J70 (35B65)},
  MRNUMBER = {969499},
MRREVIEWER = {Zuchi\ Chen},
       DOI = {10.1016/0362-546X(88)90053-3},
       URL = {https://doi.org/10.1016/0362-546X(88)90053-3},
}

@article {Ma--Yang--Yin--2025,
    AUTHOR = {Ma, Hui and Yang, Mingxuan and Yin, Jiabin},
     TITLE = {A partially overdetermined problem for the {$p$}-{L}aplace
              equation in a convex cone},
   JOURNAL = {Nonlinear Anal.},
  FJOURNAL = {Nonlinear Analysis. Theory, Methods \& Applications. An
              International Multidisciplinary Journal},
    VOLUME = {253},
      YEAR = {2025},
     PAGES = {Paper No. 113743, 17},
      ISSN = {0362-546X,1873-5215},
   MRCLASS = {35N25 (31B15 35A23 35J92 53C24)},
  MRNUMBER = {4845502},
       DOI = {10.1016/j.na.2024.113743},
       URL = {https://doi.org/10.1016/j.na.2024.113743},
}

@article {Mantoulidis--Miao--Tam--2020,
    AUTHOR = {Mantoulidis, Christos and Miao, Pengzi and Tam, Luen-Fai},
     TITLE = {Capacity, quasi-local mass, and singular fill-ins},
   JOURNAL = {J. Reine Angew. Math.},
  FJOURNAL = {Journal f\"ur die Reine und Angewandte Mathematik. [Crelle's
              Journal]},
    VOLUME = {768},
      YEAR = {2020},
     PAGES = {55--92},
      ISSN = {0075-4102,1435-5345},
   MRCLASS = {53C20 (83C40 83C57)},
  MRNUMBER = {4168687},
MRREVIEWER = {Xiaodong\ Wang},
       DOI = {10.1515/crelle-2019-0040},
       URL = {https://doi.org/10.1515/crelle-2019-0040},
}

@article {Marquardt--2017,
    AUTHOR = {Marquardt, Thomas},
     TITLE = {Weak solutions of inverse mean curvature flow for
              hypersurfaces with boundary},
   JOURNAL = {J. Reine Angew. Math.},
  FJOURNAL = {Journal f\"ur die Reine und Angewandte Mathematik. [Crelle's
              Journal]},
    VOLUME = {728},
      YEAR = {2017},
     PAGES = {237--261},
      ISSN = {0075-4102,1435-5345},
   MRCLASS = {53C44 (35J25 35J62)},
  MRNUMBER = {3668996},
MRREVIEWER = {Julian\ Scheuer},
       DOI = {10.1515/crelle-2014-0116},
       URL = {https://doi.org/10.1515/crelle-2014-0116},
}

@article {Mazurowski--Yao--2025,
    AUTHOR = {Mazurowski, Liam and Yao, Xuan},
     TITLE = {Monotone quantities for {$p$}-harmonic functions and the sharp
              {$p$}-{P}enrose inequality},
   JOURNAL = {Math. Res. Lett.},
  FJOURNAL = {Mathematical Research Letters},
    VOLUME = {32},
      YEAR = {2025},
    NUMBER = {3},
     PAGES = {957--995},
      ISSN = {1073-2780,1945-001X},
   MRCLASS = {53C21 (58J05 83C40)},
  MRNUMBER = {4945107},
MRREVIEWER = {Kai\ Xu},
       DOI = {10.4310/mrl.250731120814},
       URL = {https://doi.org/10.4310/mrl.250731120814},
}

@article {Mazurowski--Yao--2026,
    AUTHOR = {Mazurowski, Liam and Yao, Xuan},
     TITLE = {Mass, conformal capacity, and the volumetric {P}enrose
              inequality},
   JOURNAL = {Comm. Anal. Geom.},
  FJOURNAL = {Communications in Analysis and Geometry},
    VOLUME = {33},
      YEAR = {2026},
    NUMBER = {10},
     PAGES = {2355--2394},
      ISSN = {1019-8385,1944-9992},
   MRCLASS = {53C20 (53C21)},
  MRNUMBER = {5063961},
       DOI = {10.4310/cag.260419225838},
       URL = {https://doi.org/10.4310/cag.260419225838},
}

@article {Miao--2025,
    AUTHOR = {Miao, Pengzi},
     TITLE = {Mass, capacitary functions, and the mass-to-capacity ratio},
   JOURNAL = {Peking Math. J.},
  FJOURNAL = {Peking Mathematical Journal},
    VOLUME = {8},
      YEAR = {2025},
    NUMBER = {2},
     PAGES = {351--404},
      ISSN = {2096-6075,2524-7182},
   MRCLASS = {53C21 (53C20 53C80 83C40)},
  MRNUMBER = {4910987},
       DOI = {10.1007/s42543-023-00071-7},
       URL = {https://doi.org/10.1007/s42543-023-00071-7},
}

@article {Munteanu--Wang--2023,
    AUTHOR = {Munteanu, Ovidiu and Wang, Jiaping},
     TITLE = {Comparison theorems for 3{D} manifolds with scalar curvature
              bound},
   JOURNAL = {Int. Math. Res. Not. IMRN},
  FJOURNAL = {International Mathematics Research Notices. IMRN},
      YEAR = {2023},
    NUMBER = {3},
     PAGES = {2215--2242},
      ISSN = {1073-7928,1687-0247},
   MRCLASS = {53C21},
  MRNUMBER = {4565611},
MRREVIEWER = {Xiaodong\ Wang},
       DOI = {10.1093/imrn/rnab307},
       URL = {https://doi.org/10.1093/imrn/rnab307},
}

@article {Oronzio--2025,
    AUTHOR = {Oronzio, Francesca},
     TITLE = {A{DM} mass, area and capacity in asymptotically flat
              3-manifolds with nonnegative scalar curvature},
   JOURNAL = {Commun. Contemp. Math.},
  FJOURNAL = {Communications in Contemporary Mathematics},
    VOLUME = {27},
      YEAR = {2025},
    NUMBER = {9},
     PAGES = {Paper No. 2550011, 35},
      ISSN = {0219-1997,1793-6683},
   MRCLASS = {53C21 (31C12 53C20)},
  MRNUMBER = {4942308},
MRREVIEWER = {Huaiyu\ Zhang},
       DOI = {10.1142/S0219199725500117},
       URL = {https://doi.org/10.1142/S0219199725500117},
}

@article {Pigola--Setti--Troyanov--2014,
    AUTHOR = {Pigola, Stefano and Setti, Alberto G. and Troyanov, Marc},
     TITLE = {The connectivity at infinity of a manifold and
              {$L^{q,p}$}-{S}obolev inequalities},
   JOURNAL = {Expo. Math.},
  FJOURNAL = {Expositiones Mathematicae},
    VOLUME = {32},
      YEAR = {2014},
    NUMBER = {4},
     PAGES = {365--383},
      ISSN = {0723-0869,1878-0792},
   MRCLASS = {53C21 (31C12)},
  MRNUMBER = {3279484},
MRREVIEWER = {Qihua\ Ruan},
       DOI = {10.1016/j.exmath.2013.12.006},
       URL = {https://doi.org/10.1016/j.exmath.2013.12.006},
}

@book {Salsa--2016,
    AUTHOR = {Salsa, Sandro},
     TITLE = {Partial differential equations in action},
    SERIES = {Unitext},
    VOLUME = {99},
   EDITION = {Third},
      NOTE = {From modelling to theory,
              La Matematica per il 3+2},
 PUBLISHER = {Springer, [Cham]},
      YEAR = {2016},
     PAGES = {xviii+686},
      ISBN = {978-3-319-31237-8; 978-3-319-31238-5},
   MRCLASS = {35-01 (35D30 35F16 35F61 35J20 35K20 35K57 35L04)},
  MRNUMBER = {3497072},
       DOI = {10.1007/978-3-319-31238-5},
       URL = {https://doi.org/10.1007/978-3-319-31238-5},
}

@article {Silva--2025,
    AUTHOR = {Silva, Daniel},
     TITLE = {On the capacity of surfaces in asymptotically flat half-space},
   JOURNAL = {Lett. Math. Phys.},
  FJOURNAL = {Letters in Mathematical Physics},
    VOLUME = {115},
      YEAR = {2025},
    NUMBER = {2},
     PAGES = {Paper No. 42, 19},
      ISSN = {0377-9017,1573-0530},
   MRCLASS = {53C21 (53E10 83C40)},
  MRNUMBER = {4895924},
       DOI = {10.1007/s11005-025-01928-x},
       URL = {https://doi.org/10.1007/s11005-025-01928-x},
}

@article {Stern--2022,
    AUTHOR = {Stern, Daniel L.},
     TITLE = {Scalar curvature and harmonic maps to {$S^1$}},
   JOURNAL = {J. Differential Geom.},
  FJOURNAL = {Journal of Differential Geometry},
    VOLUME = {122},
      YEAR = {2022},
    NUMBER = {2},
     PAGES = {259--269},
      ISSN = {0022-040X,1945-743X},
   MRCLASS = {58E20 (53C21 53C43)},
  MRNUMBER = {4516941},
MRREVIEWER = {Radu\ Slobodeanu},
       DOI = {10.4310/jdg/1669998185},
       URL = {https://doi.org/10.4310/jdg/1669998185},
}

@article {Munteanu--Wang--2026,
    AUTHOR = {Munteanu, Ovidiu and Wang, Jiaping},
     TITLE = {Geometry of three-dimensional manifolds with positive scalar
              curvature},
   JOURNAL = {Amer. J. Math.},
  FJOURNAL = {American Journal of Mathematics},
    VOLUME = {148},
      YEAR = {2026},
    NUMBER = {1},
     PAGES = {131--160},
      ISSN = {0002-9327,1080-6377},
   MRCLASS = {53C21},
  MRNUMBER = {5025949},
       DOI = {10.1353/ajm.2026.a980770},
       URL = {https://doi.org/10.1353/ajm.2026.a980770},
}

@article {Xia--Yin--Zhou--2024,
    AUTHOR = {Xia, Chao and Yin, Jiabin and Zhou, Xingjian},
     TITLE = {New monotonicity for {$p$}-capacitary functions in 3-manifolds
              with nonnegative scalar curvature},
   JOURNAL = {Adv. Math.},
  FJOURNAL = {Advances in Mathematics},
    VOLUME = {440},
      YEAR = {2024},
     PAGES = {Paper No. 109526, 40},
      ISSN = {0001-8708,1090-2082},
   MRCLASS = {53C20 (53C21)},
  MRNUMBER = {4708141},
MRREVIEWER = {Maria\ Andrade},
       DOI = {10.1016/j.aim.2024.109526},
       URL = {https://doi.org/10.1016/j.aim.2024.109526},
}

@article{yin2026sharp,
  title={Sharp upper bounds of p-capacity in convex cone and asymptotically flat half-space},
  author={Yin, Jiabin and Zhou, Xingjian},
  journal={Letters in Mathematical Physics},
  volume={116},
  number={2},
  pages={34},
  year={2026},
  publisher={Springer}
}

\end{document}